%% file: nondegenerate_end.tex
\documentclass[notitlepage,oneside,a4paper]{amsart}
\usepackage{amsmath,amssymb,enumerate}
\usepackage{epsfig,fancyhdr,color}
\usepackage{epstopdf}
\usepackage{amsthm}
\usepackage{mathtools}
\usepackage{graphicx}
\usepackage{tikz}
\usepackage{tikz-cd}
\usepackage{cite}
\usepackage{hyperref}
\usepackage{enumitem}

\usepackage[totalwidth=16cm,totalheight=24cm]{geometry}

\input{commands.tex}

\title{Bending parameterization of one-sided degenerate Kleinian surface groups}

\author{Bruno Dular}
\address{Bruno Dular:
University of Luxembourg, FSTM, Department of Mathematics, 
Maison du nombre, 6 avenue de la Fonte,
L-4364 Esch-sur-Alzette, Luxembourg}
\email{bruno.dular@uni.lu}


\date{v3, \today}

\begin{document}
\begin{abstract}
	It was recently proved that quasi-Fuchsian manifolds are uniquely determined by their bending laminations. This paper concerns a similar result for certain non-quasi-Fuchsian manifolds: those obtained by degenerating one end but not the other, i.e.\ those appearing in boundaries of Bers slices. More precisely, we show that such hyperbolic manifolds are uniquely determined by the end structure of the degenerate end and the bending lamination of the other. The end structure consists of the parabolic locus, which is a multicurve, together with ending laminations or conformal structures on each component of its complement.
\end{abstract}

\maketitle

\section{Introduction}
This paper focuses on $3$-dimensional hyperbolic manifolds (complete Riemannian $3$-manifolds locally isometric to the $3$-dimensional hyperbolic space $\HH^3$) that are homeomorphic to $M\coloneqq S\times \R$, where $S=S^+$ is a closed oriented surface of genus $\geq 2$, and $S^-$ is the same surface with reversed orientation. Let $\Gamma\coloneqq\pi_1(M)=\pi_1(S)$. Such a manifold is either \emph{convex-cocompact}, it is then called \emph{quasi-Fuchsian}, or not, in which case it can be obtained as a \emph{limit} of quasi-Fuchsian ones, by the density theorem \cite{MR2079598,MR2821565,MR3001608}.

To a $3$-dimensional hyperbolic manifold $M$ homeomorphic to $S\times\R$, one can associate the action $\rho\colon\Gamma\to\PSL(2,\C)\cong\Isom^+(\HH^3)$ of its fundamental group on its universal covering $\widetilde{M}\cong\HH^3$ by isometries, which is a discrete and faithful representation, and $M$ is homeomorphic to $M_{\rho}\coloneqq\HH^3/\rho(\Gamma)$. There are two central objects for studying the geometry of $M$: its conformal boundary, or boundary at infinity, and its convex core. The action $\rho$ extends to an action by conformal transformations of the conformal boundary $\partial_{\infty}\HH^3\cong S^2$. Its limit set $\Lambda(\rho)$ is a closed subset of $S^2$, with complement its domain of discontinuity $\Omega(\rho)$. The quotient $\partial_{\infty}\rho\coloneqq\Omega(\rho)/\rho(\Gamma)$ is a finite union of Riemann surfaces of finite type \cite{MR167618}. It is the \emph{conformal boundary} of $M$. The \emph{convex core} $C(M)=C(\rho)$ is the quotient of the convex hull of $\Lambda(\rho)$ in $\HH^3$ by the action $\rho$. It is a convex submanifold of $M$, and is actually the smallest convex submanifold whose inclusion $C(M)\monic M$ is a homotopy equivalence.

The \emph{algebraic deformation space} of hyperbolic structures on $S\times\R$, denoted $\cAH(S)$, is the space of discrete and faithful representations $\rho\colon\Gamma\to\PSL(2,\C)$, up to conjugation. Depending on the context, an element $\rho\in\cAH(S)$ is identified with (the conjugacy class of) its image $\rho(\Gamma)\subset\PSL(2,\C)$, called a \emph{Kleinian surface group}, or with the hyperbolic manifold $M_{\rho}\coloneqq\HH^3/\rho(\Gamma)$, a \emph{hyperbolic structure} on $S\times\R$. The fact that $M_{\rho}$ is homeomorphic to $S\times\R$ follows from Bonahon's tameness theorem~\cite{MR847953} (see \S~\ref{sec:background}).

The interior of $\cAH(S)$ is denoted $\cQF(S)$ and consists of \emph{quasi-Fuchsian} representations, characterized by the compactness of their convex core. Given $\rho\in\cQF(S)$, $\Omega(\rho)$ is topologically the union of two disks and $\partial_{\infty}\rho$ is a Riemann surface structure $(\mm^-(\rho),\mm^+(\rho))$ on $S^-\sqcup S^+$. By the simultaneous uniformization theorem of Ahlfors--Bers \cite{MR111834}, the induced map
\begin{equation}\label{eq:Ahlfors--Bers}
	\mm^-\times \mm^+\colon\cQF(S)\to\cT(S^-)\times\cT(S^+)
\end{equation}
is a biholomorphic homeomorphism, whose inverse is denoted $\UU$. Let $\cF(S)\subset\cQF(S)$ be the subspace of \emph{Fuchsian} structures, i.e.\ those $\rho$ with image in $\PSL(2,\R)$, up to conjugation.

Given a non-Fuchsian quasi-Fuchsian structure $\rho\in\cQF(S)-\cF(S)$, its convex core $C(\rho)$ is homeomorphic to $S\times [0,1]$, so its boundary $\partial C(\rho)$ is also (topologically) a copy of $S^-\sqcup S^+$, but now lying as a \emph{pleated surface} inside $M_{\rho}$, i.e.\ a piecewise totally geodesic subsurface \cite{thurston1979geometry,MR2230672}. Its \emph{pleating} or \emph{bending} is recorded by a pair of measured laminations $(\bb^-(\rho),\bb^+(\rho))\in\cML(S^-)\times\cML(S^+)$: the pleating locus constitutes a geodesic lamination and the bending angles, or limits thereof, a transverse measure. The map
$$\bb^-\times\bb^+\colon\cQF(S)\to\cML(S^-)\times\cML(S^+)$$
is continuous \cite{MR1331998,MR1678469,MR2464096}. Bonahon and Otal determine its image in \cite{MR2144972} (in the more general case of geometrically finite structures, and later generalized to the compressible case by Lecuire \cite{MR2207784}): it is the subspace $\cML^{\perp}_{<\pi}(S^-\sqcup S^+)$ of filling pairs\footnote{A pair $(\mu^-,\mu^+)\in\cML(S^-)\times\cML(S^+)$ is \emph{filling} if, after identifying $S^-$ and $S^+$, any nonzero measured lamination $\lambda$ on $S$ satisfies $i(\lambda,\mu^-)+i(\lambda,\mu^+)>0$.} $(\mu^-,\mu^+)$ without closed leaves of weight $\geq\pi$. Properness of the map $\bb^-\times\bb^+$ onto the subspace $\cML^{\perp}_{<\pi}(S^-\sqcup S^+)$ is due to Lecuire \cite{lecuire2025propernessbendingmap}.

In ~\cite{MR2144972}, Bonahon and Otal show that $\bb^-\times\bb^+$ is injective over \emph{rational laminations}, i.e.\ weighted multicurves. Bonahon also showed that $\bb^-\times\bb^+$ is injective in a neighborhood of Fuchsian structures inside $\cQF(S)$ \cite{MR2186972}. Keen and Series proved injectivity when $S$ is the punctured torus \cite{MR2052972,MR2258745}. Dular and Schlenker showed in \cite{dularschlenker2024pleating} that $\bb^-\times\bb^+$ is injective on $\cQF(S)-\cF(S)$.\footnote{Note that the result of \cite{dularschlenker2024pleating} only applies when $S$ is closed, while Bonahon's result \cite{MR2186972} (and also Keen--Series) also covers non-closed surfaces. The results of the present paper concern closed surfaces.} It follows that the map
$$\bb^-\times\bb^+\colon\cQF(S)-\cF(S)\to\cML^{\perp}_{<\pi}(S^-\sqcup S^+)$$
is a homeomorphism.

Using the same methods as in \cite{dularschlenker2024pleating}, we first give a mix of the above two parameterizations of quasi-Fuchsian structures, which will be a key step in the proof of the main result. Let $\cML_{<\pi}(S^+)$ be the space of measured laminations on $S^+$ without closed leaves of weight $\geq\pi$.
\begin{theorem*}[Theorem \ref{thm:prescribed_metric_and_bending} in the text]
	Consider the map $\mm^-\times\bb^+\colon\cQF(S)\to\cT(S^-)\times\cML_{<\pi}(S^+)$ sending a quasi-Fuchsian structure $\rho$ to the pair $(\mm^-(\rho),\bb^+(\rho))$ of the conformal structure of the ideal boundary of its bottom end and the bending lamination of the top boundary of its convex core. It is a homeomorphism.
\end{theorem*}

The main result concerns certain hyperbolic structures on $S\times\R$ which are not quasi-Fuchsian, namely those obtained by partially degenerating one end but not the other. In this paper, we will always degenerate the bottom end. Typically, such degeneration occurs at the boundary of \emph{Bers slices}, i.e.\ $m^+\in\cT(S^+)$ remains fixed while $\mm^-$ exits all compact subsets of $\cT(S^-)$. One can pinch a multicurve $P$ on $S^-$, in which case rank $1$ cusps appear in the three-manifold. Individual components of the complement $S^--P$, which are cusped surfaces, can themselves degenerate in a way that is recorded into an \emph{ending lamination}. Components of $S^-$ that stay non-degenerate form the \emph{moderate subsurface} of $S^-$, while the others form the \emph{immoderate subsurface}. Let $\cAH^+(S)\subset\cAH(S)$ denote the space of structures whose top end is non-degenerate, i.e.\ whose top end is geometrically finite without cusps. They are referred to as \emph{one-sided degenerate structures}.

\begin{figure}[h]
	\centering
	\begin{tikzpicture}
		\node[anchor=south west] at (0,0){\includegraphics[width=0.8\textwidth]{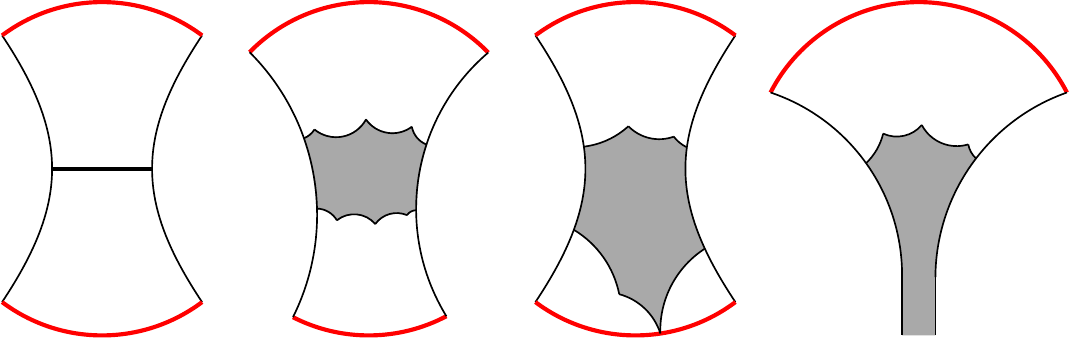}};
		\node[align=center] at (1.25,4.35) {\tiny $\partial_{\infty}^+ M$};
		\node[align=center] at (1.25,-0.1) {\tiny $\partial_{\infty}^- M$};
		\node[align=center] at (1.3,2.35) {\small $C(M)$};
		\node[align=center] at (1.35,1.9) {\small $\cong S$};
		\node[align=center] at (4.5,2.1) {\small $C(M)$};
		\node[align=center] at (4.5,2.9) {\tiny $\partial^+ C(M)$};
		\node[align=center] at (4.5,1.3) {\tiny $\partial^- C(M)$};
		\node[align=center] at (8.1,-0.12) {\small cusp};
		\node[align=center] at (11.1,-0.12) {\small degenerate end};
	\end{tikzpicture}
	\caption{Schematic representation of hyperbolic manifolds homeomorphic to $S\times\R$, with $1$-dimensional ``$S$-direction'', conformal boundary in red and convex core in grey. The structures are, from left to right: Fuchsian, quasi-Fuchsian, Kleinian with a cusp and Kleinian with degenerate bottom end.}
\end{figure}
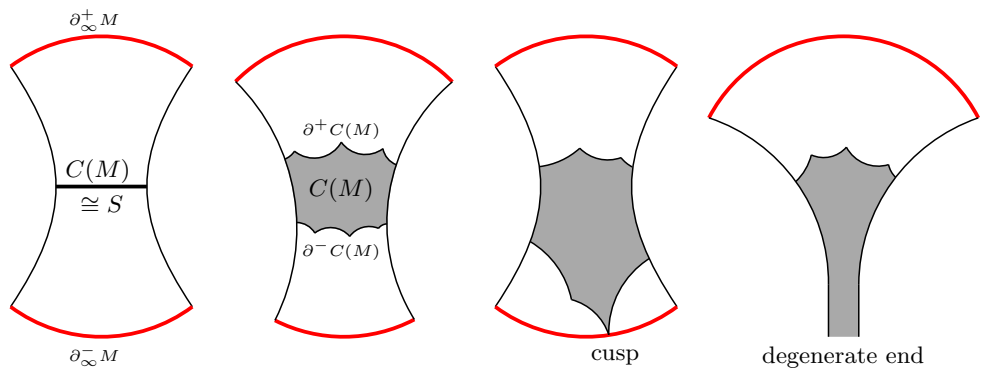

By the theory of quasi-conformal deformations of Kleinian groups (Ahlfors, Bers, Kra, Marden, Maskit, Sullivan \cite[Theorem 7.55]{MR1638795}) together with the ending lamination theorem for Kleinian surface groups by Brock, Canary and Minsky \cite{MR2925381}, a Kleinian surface structure $\rho\in\cAH(S)$ is entirely determined by its \emph{end invariants} $(\EE^-(\rho),\EE^+(\rho))$, which are combinations of the data of the \emph{parabolic locus} (the simple multicurve that is pinched in $S^-\sqcup S^+$ to obtain $\partial_{\infty}\rho$), and \emph{conformal structures} at infinity or \emph{ending laminations} on each component of the complement (see \S\ref{sec:end_structures} for the notations).

In the case of one-sided degenerate structures, this says that $\rho\in\cAH^+(S)$ is uniquely determined by a pair $(\EE^-(\rho),\mm^+(\rho))$, where $\mm^+(\rho)$ is the top conformal structure at infinity and $\EE^-(\rho)$ is the end structure of the bottom end. A consequence of the main result of this paper is the analogous statement, with $\mm^+(\rho)$ replaced by the bending lamination on the top side of the convex core $\bb^+(\rho)$.

\begin{theorem*}
	Let $\rho\in\cAH^+(S)$ be a one-sided degenerate Kleinian structure on $S\times\R$, and let $M_{\rho}\coloneqq\HH^3/\rho(\pi_1(S))$. Then $\rho$ is uniquely determined by the pair $(\EE^-(\rho),\bb^+(\rho))$, i.e.\ by the bending lamination $\bb^+(\rho)$ of the top side of the convex core of $M_{\rho}$ and by the end structure $\EE^-(\rho)$ of the bottom end of $M_{\rho}$.

	More precisely, let $P\subset S^-$ be a simple multicurve, let $\Sigma^-$ be a union of connected components of $S^--P$ and let $L$ be a union of arational laminations on each component of $S^--(P\cup\Sigma^-)$. Denote by $\cAH^+(S;P\cup L)\subset\cAH(S)$ the subspace of one-sided degenerate structures with bottom end cusped along $P$ and having ending laminations $L$. Let $\cML^{\perp P\cup L}_{<\pi}(S^+)\subset\cML(S^+)$ be the subspace of measured laminations $\mu$ on $S^+$ without closed leaves of weight $\geq\pi$ and such that $\mu$ and $P\cup L$ fill $S-\Sigma^-$, i.e.\ any nonzero measured lamination on $S$ has nonzero intersection with $\mu$ or $P\cup L$ or has support intersecting $\Sigma^-$.
	
	Consider the map
	$$\mm^-\times\bb^+\colon\cAH^+(S;P\cup L)\longrightarrow\cT(\Sigma^-)\times\cML^{\perp P\cup L}_{<\pi}(S^+)$$
	sending a structure $\rho$ to the pair $(\mm^-(\rho),\bb^+(\rho))$ of the conformal structure at infinity of the moderate subsurface of the bottom end and the bending lamination of the top side of the convex core. Then this map is a homeomorphism.
\end{theorem*}

\subsection{Outline of the proof}\label{sec:outline}
Fix a simple multicurve $P\subset S^-$ and a union $L$ of arational laminations in some connected components of $S^--P$. Let $e^-=(P,L)$ and recall that $\cAH^+(S;e^-)$ denotes the space of hyperbolic structures on $S\times\R$ with non-degenerate top end, and with parabolic locus $P$ and ending lamination $L$ in the bottom end. We need to prove that the map
$$\mm^-\times\bb^+\colon\cAH^+(S;e^-)\to\cT(\Sigma^-)\times\cML_{<\pi}^{\perp e^-}(S^+)$$
is continuous, proper and injective. This implies that it is a homeomorphism by invariance of domain, since both sides are topological manifolds of the same dimension and the target is connected.

Continuity follows from Lecuire's work~\cite{MR2464096}, see Proposition \ref{prop:continuity_bending_map} in Section~\ref{sec:continuity}. Then, in Section~\ref{sec:properness} we prove properness. The argument follows closely the cases considered by Bonahon--Otal \cite{MR2144972}, Lecuire \cite{MR2207784}, Anderson--Lecuire \cite{MR3134412}, Baba--Ohshika \cite{MR4651897} and Lecuire \cite{lecuire2025propernessbendingmap}. The first step is to prove that any sequence of representations with convergent image under $\mm^-\times\bb^+$ must have an algebraically convergent subsequence, using the theory of degeneration of metrics towards actions on real trees of Morgan--Shalen \cite{MR769158,MR1402300}. Then, it follows from results of Anderson--Lecuire \cite{MR3134412} and Bonahon--Otal and Lecuire \cite{MR2144972,MR2207784} that this algebraic limit must lie in $\cAH^+(S;e^-)$.

To prove injectivity, we apply an argument similar to the proof of the injectivity of $\bb^-\times\bb^+$ for convex-cocompact structures \cite{dularschlenker2024pleating}. More precisely, we first fix a point $m^-\in\cT(\Sigma^-)$ and focus on the slice $\cAH^{+}(S;\cE^-)\subset\cAH^+(S;e^-)$ of structures $\rho$ with end structure $\EE^-(\rho)=\cE^-$ consisting of $e^-=P\cup L$ and $m^-\in\cT(\Sigma^-)$. We then show that the map
$$\bb^+_{\infty}\colon\cAH^{+}(S;\cE^-)\to\cML_{<\pi}^{\perp e^-}(S^+)$$
is injective, by exhibiting it as a limit of homeomorphisms $\bb^{+}_{n}$, $n\in\N$. Combined with the fact that $\bb^+_{\infty}$ has compact real-analytic fibres and using a theorem of Finney \cite{MR0224087}, this implies that $\bb^+_{\infty}$ is injective. 

\subsection{Strategy of the proof of injectivity}\label{sec:strategy_injectivity}
Consider a sequence $r\colon \N\to\cT(S^-)$ that converges to $\cE^-$ in the space of \emph{end structures} $\cES(S^-)$ (called space of \emph{gallimaufries} in \cite{MR3134412}, see Section \ref{sec:convergence_end_structures} for a description of its topology), in other words, $r$ extends to a continuous map $\overline{r}\colon \widehat{\N}\to\cES(S^-)$ with $\overline{r}(\infty)=\cE^-$, where $\widehat{\N}\coloneqq\N\cup\{\infty\}$ is the one-point compactification of $\N$ with the cofinite topology. Such a sequence exists by Proposition~\ref{prop:continuity_end_structures}~\cite{MR3134412} and the description of the topology in Section~\ref{sec:convergence_end_structures}.

The objects and maps involved in the proof are summarized in the diagram below, where $\overline{\UU}$ denotes the extension of $\UU$ (the inverse of the map \eqref{eq:Ahlfors--Bers}) to the space of end structures.

\[\begin{tikzcd}
	{\N\times\cT(S^+)} & {\cT(S^-)\times\cT(S^+)} & {\cQF(S)} & {\cML(S^-)\times\cML(S^+)} \\
	{\widehat{\N}\times\cT(S^+)} & {\cES(S^-)\times\cT(S^+)} & {\cAH^+(S)} & {\cML(S^+)}
	\arrow["{r\times\rm{incl}}", from=1-1, to=1-2]
	\arrow[hook, from=1-1, to=2-1]
	\arrow["\UU", from=1-2, to=1-3]
	\arrow["{\bb^-\times\bb^+}", hook, from=1-3, to=1-4]
	\arrow[hook, from=1-3, to=2-3]
	\arrow["{\rm{proj}_2}", from=1-4, to=2-4]
	\arrow["{\overline{r}\times\rm{incl}}"', from=2-1, to=2-2]
	\arrow["{\BB^+}"{description}, bend right=20, from=2-1, to=2-4]
	\arrow["{\overline{\UU}}"', from=2-2, to=2-3]
	\arrow["{\bb^+}"', from=2-3, to=2-4]
\end{tikzcd}\]

Let $\BB^+\colon \widehat{\N}\times\cT(S^+)\to\cML(S^+)$ be the composition $\BB^+\coloneqq \bb^+\circ\,\overline{\UU}\circ (\overline{r}\times\rm{incl})$, and for each $n\in \widehat{\N}$, write $\bb_n^+\coloneqq\BB^+(n,\cdot)\colon\cT(S^+)\to\cML(S^+)$.

To show that $\bb_{\infty}^+$ is injective, we need the following ingredients:
\begin{enumerate}[label=\Roman*]
	\item\label{item:continuity} The map $\BB^+$ is continuous (Corollary \ref{cor:continuity}), hence $\bb^+_{\infty}$ is the \emph{continuous} limit of $\bb^+_{n}$ when $n\to\infty$;
	\item\label{item:compact_fibres} The map $\bb^+_{\infty}$ is proper (\S \ref{sec:properness});
	\item\label{item:real_analycity} The map $\bb^+_{\infty}$ has real-analytic fibres (\S \ref{sec:analyticity});
	\item\label{item:injectivity} For each $n\in \N$, the map $\bb^+_n$ is injective (by invariance of domain, it is then an open embedding) (\S \ref{sec:injectivity}).
\end{enumerate}
Combining \eqref{item:continuity} and \eqref{item:injectivity}, Theorem \ref{thm:fibres_are_weakly_contractible}, which is a variant of Finney's theorem \cite{MR0224087,dularschlenker2024pleating}, implies that the compact ``topologically nice'' fibres of $\bb^+_{\infty}$ are contractible. Any fibre $F$ of $\bb^+_{\infty}$ is compact by \eqref{item:compact_fibres} and ``nice'' by \eqref{item:real_analycity}. All in all, $F$ is a  contractible compact real-analytic space. But a compact real-analytic space $F$ has a fundamental class modulo $2$ \cite{MR149503,MR278333}, i.e.\ a nonzero homology class in $H_{\dim F}(F;\Z/2)$. Since $F$ is contractible, we must have $\dim F=0$, which means that $F$ is a point. This shows that $\bb^+_{\infty}$ is injective.

In Section \ref{sec:background} the necessary background is recalled. Sections \ref{sec:continuity}-\ref{sec:injectivity} are devoted to proving the above four ingredients, with the conclusion of the proof of the main result at the end of Section \ref{sec:injectivity}.

\subsection{Further directions}
A natural question is whether \emph{any} structure in $\cAH(S)$ is uniquely determined by the combination of its bending and ending laminations. The main theorem of this paper implies that it is the case when the top side is non-degenerate and the bottom moderate subsurface has no quasi-conformal deformations, i.e.\ $\Sigma^-$ is empty or a union of thrice-punctured spheres. But this is still far from a complete answer.

An essential difficulty lies in proving that any \emph{geometrically finite} structure on $S\times\R$ is uniquely determined by its pair of bending laminations. This is expected to hold, and it was proved in the case of \emph{rational} bending laminations by Bonahon and Otal \cite{MR2144972}. In the general case, the current strategy needs to be refined in order to prove \emph{rigidity} of geometrically finite structures with respect to their bending lamination.

For geometrically infinite Kleinian surface structures, an explicit description of the combination of ending laminations and bending laminations that can occur was recently given by Baba and Ohshika in \cite{MR4651897}.

\subsection*{Acknowledgements}
I wish to thank Tommaso Cremaschi, Cyril Lecuire and Jean-Marc Schlenker for helpful discussions, as well as the anonymous referee for helpful comments. I was supported by FNR AFR grant 18890152 during the preparation of this paper.

\section{Background}\label{sec:background}
The main references used are the books of Matsuzaki--Taniguchi \cite{MR1638795}, Canary--Epstein--Marden \cite{MR2230672} and Marden \cite{MR3586015}.

Unless stated otherwise, $S=S^+$ denotes a closed oriented surface of genus $\geq 2$, and $S^-$ is the same surface with reversed orientation. Let $\Gamma\coloneqq\pi_1(S)$. The space of representations of $\Gamma$ into $\PSL(2,\C)\cong\Isom^+(\HH^3)$, modulo conjugation, is denoted $\cX(S)$. Let $\cAH(S)\subset\cX(S)$ denote the subspace of discrete and faithful representations. Any $\rho\in\cAH(S)$ determines a hyperbolic manifold $M_{\rho}\coloneqq\HH^3/\rho(\Gamma)$. By Bonahon's tameness theorem \cite[Corollary B]{MR847953}, $M_{\rho}$ is homeomorphic to the interior of a compact $3$-manifold, which then must be homeomorphic to $S\times [0,1]$ by \cite[10.6]{MR415619} (see also Scott's compact core theorem~\cite{MR326737}).

Thus $\cAH(S)$ is also called the \emph{algebraic deformation space of marked hyperbolic structures on $S\times\R$}.

\begin{remark}
	Let $\rho\in\cAH(S)$. Up to isotopy, there are exactly two homeomorphisms between $M_{\rho}$ and $S^+\times\R$, one preserves the orientation and the other one reverses it (and they are homotopic to each other). When writing $M_{\rho}\cong S\times\R$, the homeomorphism is assumed to be orientation-preserving. In particular, this allows speaking of the \emph{bottom} and \emph{top} ends of a Kleinian surface manifold. See \cite[\S 5.2.1]{MR3586015}.
\end{remark}

Given $\rho\in\cAH(S)$ and $c\subset S$ a curve on $S$, we denote by $\ell_{\rho}(c)$ the length of a geodesic representative of $c$ inside $M_{\rho}$, or zero if $c$ has no such geodesic representative. Equivalently, if $\gamma\in\pi_1(S)$ represents $c$, then $\ell_{\rho}(c)$ is the translation distance of the isometry $\rho(\gamma)$ on $\HH^3$.

\subsection{Laminations and trees}\label{sec:laminations_and_trees}
Let $\cGL(S)$ be the space of \emph{geodesic laminations} on $S$, i.e.\ closed subsets that are the union of disjoint simple geodesics (its \emph{leaves}), endowed with the Chabauty topology (which coincides with the Hausdorff topology and geometric topology \cite[\S I.3.1]{MR2230672}). A geodesic lamination is said to be \emph{minimal} or \emph{maximal} when it is so with respect to inclusion (among non-empty laminations). A minimal geodesic lamination is either a simple closed curve or it is irrational, in which case it contains uncountably many geodesics and any half-leaf is dense in it. A leaf of a geodesic lamination $L$ is \emph{recurrent} if it is contained in a minimal sublamination. A geodesic lamination is \emph{recurrent} if all its leaves are. A geodesic lamination is \emph{rational} if it is a simple multicurve, i.e.\ a finite union of simple closed curves, and \emph{arational} if it is minimal, irrational and each component of its complement is either simply-connected or an annulus around a cusp. The space $\cGL(S)$ is compact \cite[I.4.1.7]{MR2230672}.

A \emph{transverse measure} on a geodesic lamination $L\in\cGL(S)$ is a map $\mu$ associating to each closed arc transverse to $L$ a number $\mu(k)\in\R_{\geq 0}$ (also denoted $i(\mu,k)$) which is invariant under transverse isotopy of arcs, countably additive, and with full support, i.e.\ $\mu(k)>0$ if and only if $k\cap L\not =\emptyset$. Let $\cML(S)$ denote the space of measured laminations on $S$, endowed with the weak topology. The geodesic laminations that admit a transverse measure (with full support) are precisely the recurrent laminations \cite[\S 1.7]{MR1144770}, and the transverse measures they admit are not unique up to scaling in general (see e.g.\ \cite{MR391184,MR409766,MR2470969}). See \cite{MR1144770} for more details on measured laminations, and a description of a piecewise linear structure on $\cML(S)$ using train tracks. By Thurston's earthquake theorem \cite{MR903860}, the space $\cML(S)$ is homeomorphic to $\cT(S)$. See also \cite{MR2230672}, \cite[Appendix]{MR1402300}, \cite[\S 6.1]{MR1638795}.

Denote by $i\colon\cML(S)\times\cML(S)\to\R_{\geq 0}$ the \emph{intersection form} between measured laminations, which continuously extends both the intersection number between weighted multicurves and the measure of transverse arcs.

To a measured lamination $\mu$ on $S$ one can associate its dual \emph{real tree} $\cT_{\mu}$, together with an action $\rho$ of $\pi_1(S)$ on $\cT_{\mu}$ by isometries \cite[\S 8.2]{MR1402300}. For $\gamma\in\pi_1(S)$, the translation distance of $\gamma$ on $\cT_{\mu}$ is precisely $i(\mu,\gamma)$. A sequence of representations $(\rho_n)_{n\in\N}\subset\cAH(S)$ is said to \emph{converge} to an action $\rho$ of $\pi_1(S)$ on a real tree $\cT_{\mu}$ if there is a sequence of positive real numbers $\epsilon_n\searrow 0$ such that $\epsilon_n\ell_{\rho_n}(\gamma)\to i(\mu,\gamma)$ for all $\gamma\in\pi_1(S)$.

Morgan and Shalen proved that actions on real trees arise naturally as limits of divergent sequences of discrete and faithful representations of $\pi_1(S)$ in $\Isom^+(\HH^3)$, as stated below \cite{MR769158}. The fact that real trees arising in this way are in fact dual to measured laminations is due to Skora \cite{MR1339846}. See \cite{MR1402300} for a detailed exposition.

\begin{theorem}[Morgan--Shalen \cite{MR769158,MR1339846,MR1402300}]\label{thm:Morgan-Shalen}
	Let $(\rho_n)_{n\in\N}$ be a sequence in $\cAH(S)$. Then, up to extracting a subsequence, either:
	\begin{enumerate}
		\item The sequence $(\rho_n)_n$ converges in $\cAH(S)$, or;
		\item The sequence $(\rho_n)_n$ converges to an action $\rho$ of $\pi_1(S)$ on a real tree $\cT_{\mu}$, where $\mu$ is a nonzero measured lamination on $S$.
	\end{enumerate}
\end{theorem}

A geodesic lamination $\alpha$ is said to be \emph{realized} in the real tree $\cT_{\mu}$ if all its leaves intersect $\mu$ transversely \cite[\S 3.1]{MR1402300}. See \cite{MR769158}, \cite{MR1339846} and \cite{MR1402300} for more details.

\subsection{Pleated surfaces}
Let $M$ be a hyperbolic three-manifold, $S$ a complete hyperbolic surface (not necessarily compact) and $L$ a geodesic lamination on $S$. A \emph{pleated surface with pleating locus} $L$ is a continuous map $f\colon S\to M$ which is a totally geodesic immersion on $S-L$ and on $L$. In general, the \emph{flat pieces} of $S-L$ are bent along the pleating locus, with an angle recorded into a \emph{bending cocycle} $\beta$ which is an element of the space $\cH(L;\R/2\pi\Z)$ of $\R/2\pi\Z$-valued cocycles along $L$ \cite{MR1413855} (which is a torus whose dimension is determined by $S$ and $L$). When $f$ is locally convex, this bending cocycle defines a transverse measure with support contained in $L$, called the \emph{bending} or \emph{pleating (measured) lamination} of $f$. The induced path metric on $S$ inside $M$ is hyperbolic and defines a point $h\in\cT(S)$, which can also be encoded in a \emph{shearing cocycle} $\sigma\in\cH(L;\R)$. The boundary of the convex core $C(\rho)$ of $\rho\in\cAH(S)$ is an embedded convex pleated surface \cite[II.1.11,12]{MR2230672}. In particular, its bending measured lamination has weight $<\pi$ along closed leaves.

Now, let $\Lambda$ be a \emph{maximal} geodesic lamination. Let $\rho\in\cAH(S)$. If there exists a pleated surface inside $M_{\rho}$ with pleating locus $\Lambda$ (we say that $\Lambda$ is \emph{realized} in $M_{\rho}$ or by $\rho$), then it is unique and the \emph{shear-bend cocycle} provides a parameterization of such representations. More generally, one can consider \emph{abstract pleated surfaces} $(\widetilde{f},\pi_1(S),\rho)$ where $\widetilde{f}\colon\widetilde{S}\to\HH^3$ is a pleated surface inside $\HH^3$ which is equivariant with respect to the (not necessarily discrete) representation $\rho\in\cX(S)$. Given a maximal geodesic lamination $\Lambda$, let $\cR(\Lambda)$ denote the open subspace of $\cX(S)$ consisting of those representations \emph{realizing} $\Lambda$.

\begin{theorem}[{Bonahon \cite[Theorem D]{MR1413855}}]\label{thm:shear_bend_bonahon}
	Let $\Lambda$ be a maximal geodesic lamination. The map $\cR(\Lambda)\to\cH(\Lambda;\C/2\pi i\Z)$ sending a representation $\rho$ realizing $\Lambda$ to its shear-bend cocycle $\sigma(\rho)+i\beta(\rho)$ is a biholomorphic homeomorphism onto an open subset of $\cH(\Lambda;\C/2\pi i\Z)$.
\end{theorem}

Using pleated surfaces one can also speak of the length of a \emph{realizable} measured lamination in a Kleinian surface manifold. Consider $\lambda\in\cML(S)$ and $\rho\in\cAH(S)$. If $\rho$ realizes $|\lambda|$ (if $|\lambda|$ is non-maximal, this means that $M_{\rho}$ contains a pleated surface $P$ whose pleating locus contains the support $|\lambda|$ of $\lambda$), then the length $\ell_{\rho}(\lambda)$ is defined as the length of $\lambda$ on $P$. When $\lambda$ is a closed curve, this coincides with the length of a geodesic representative. The following was stated by Thurston and proved by Brock.

\begin{theorem}[Brock \cite{MR1791139}]
	The length function $(\rho,\lambda)\mapsto\ell_{\rho}(\lambda)$ is continuous on the open subset of $\cAH(S)\times\cML(S)$ consisting of pairs $(\rho,\lambda)$ such that $\rho$ realizes $|\lambda|$.
\end{theorem}

\subsection{End structures}\label{sec:end_structures}
Let $M$ be a hyperbolic three-manifold with incompressible boundary, and let $C(M)$ be its convex core. See \cite[\S 3.12,5.5]{MR3586015} for a definition of the \emph{relative ends} of $M$, which describe an \emph{ideal boundary} of $M$. A relative end $E$ corresponds to a decreasing nested sequence $V_1\supset V_2\supset\dots$ of connected open subsets of $M$ ``going to infinity'', each homeomorphic to $S_E\times (0,1)$, where $S_E$ is a finite type oriented surface (possibly with cusps) that only depends on $E$ (and actually realized as a pleated surface inside $M$).

A relative end is either contained in the complement of the convex core $C(M)$, in which case it is \emph{geometrically finite} and it corresponds to a component of $\partial_{\infty}M$, or it is contained in $C(M)$, it is then \emph{geometrically infinite}. Geometrically finite ends do not contain any closed geodesic, since all closed geodesics are contained in the convex core $C(M)$. On the other hand, a geometrically infinite end $E$, with associated surface $S_E$, is characterized by the existence of a sequence of simple closed geodesics $(\gamma^*_n)_n$ in $S_E$ whose geodesic representatives $(\gamma_n)_n$ in $M$ have bounded length and exit the end $E$. The limit of $(\gamma_n^*)_n$ is a well-defined recurrent (equivalently, measurable) geodesic lamination $\Lambda(E)\subset S_E$, called the \emph{ending lamination} of $E$ \cite{thurston1979geometry,MR847953}. It is an arational recurrent geodesic lamination on $S_E$.

The \emph{end invariant} of a Kleinian surface group $\rho\in\cAH(S)$ can be encoded into a quadruplet $\EE(\rho)=(P,\Sigma,m,L)$ where the parabolic locus $P$ is a simple essential multicurve on $S^-\sqcup S^+$, the moderate part $\Sigma$ is a union of connected components of $(S^-\sqcup S^+)-P$, $m\in\cT(\Sigma)$ and $L$ is a union of arational recurrent laminations in each component of $S(L)\coloneqq (S^-\sqcup S^+)-(P\cup\Sigma)$. For a connected component $L'$ of $L$, let $S(L')$ denote the connected component of $S(L)$ containing $L'$. In particular the end invariant $\EE(\rho)$ decomposes into a \emph{topological end data} $e=(P,L)$ and a \emph{conformal end data} $c=(\Sigma,m)$. We also identify $e$ with the measurable geodesic lamination $P\cup L$ and denote by $S(e)$ the open subsurface of $S$ obtained as the union of $S(L)$ and collar neighborhoods of components of $P$.

The end structure $\EE(\rho)$ is always \emph{doubly incompressible} in the sense of \cite[\S 2.9]{MR3134412}, i.e.\ $P\cup L$ has no parallel components. Any such end structure is realizable as the end invariant of a Kleinian surface group \cite{MR2821565,MR3001608}.

Given a topological end datum $e=(P,L)$, we denote by $\cAH(S;e)$ the subspace of $\cAH(S)$ consisting of those representations $\rho$ whose end structure $\EE(\rho)$ has topological end datum $e$. Denote also by $e^-$ and $e^+$ the \emph{bottom} and \emph{top} parts of the topological end datum, respectively.

For a one-sided degenerate structure $\rho\in\cAH^+(S)$, its parabolic locus $P$ is entirely contained in $S^-$ and its moderate part has the form $\Sigma^-\sqcup S^+\subseteq S^-\sqcup S^+$.

Baba and Ohshika described the combination of ending laminations, parabolic locus and bending laminations that occur for Kleinian surface groups \cite{MR4651897}. It follows from their result that, for $\rho\in\cAH^+(S;e^-)$ where $e^-=e=P\cup L$, the bending lamination $\bb^+(\rho)$ of the top boundary of its convex core \emph{fills} $S(e^-)$ with $e^-$, i.e.\ any nonzero measured lamination on $S(e^-)$ has nonzero intersection with $P\cup L$ or $\bb^+(\rho)$. Hence it lies in the subspace $\cML^{\perp e^-}_{<\pi}(S^+)$ consisting of those measured laminations without closed leaves of weight $\geq\pi$, and which fill $S(e^-)$ with $e^-$. They also proved that every element of $\cML^{\perp e^-}_{<\pi}(S^+)$ can be realized as the top bending lamination of a structure in $\cAH^+(S;e^-)$.

\subsection{Metrics converging to an end structure}\label{sec:convergence_end_structures}
In \cite{MR3134412}, Anderson and Lecuire define a natural topology on the space $\cES(F)$ of end structures on a hyperbolizable surface $F$, i.e.\ of doubly incompressible quadruplets $(P,\Sigma,m,L)$ on $F$ as in the previous section (with $F$ instead of $S^+\sqcup S^-$).

For our purpose we only need to describe one particular type of convergence. A sequence of hyperbolic metrics $m_n\in\cT(F)$ (seen as end structures of the form $(\emptyset,F,m_n,\emptyset)$) converges to an end structure $(P,\Sigma,m,L)$ if the following two conditions hold:
\begin{enumerate}
	\item The metrics $m_n\restr{\Sigma}$ converge to $m$ in the length spectrum, i.e.\ $\ell_{m_n}(\partial\overline{\Sigma})\to 0$ and $\ell_{m_n}(c)\to\ell_m(c)$ for every non-peripheral closed curve $c\subset\Sigma$.
	\item For each component $L'$ of $L$, the restrictions $m_n\restr{S(L')}$ converge to $L'$, i.e.\ there is a non-peripheral simple closed curve $c\subset S(L')$ such that
	$\ell_{m_n}(\partial\overline{S(L')})/\ell_{m_n}(c)\to 0$
	and any subsequence of $m_n\restr{S(L')}$ contains a subsequence converging to a projective measured lamination supported by $L'$.
\end{enumerate}
Thus, we can identify $\cT(F)$ with an open subspace of $\cES(F)$. Note that $\cES(S^-\sqcup S^+)=\cES(S^-)\times\cES(S^+)$.

The \emph{end invariants} of an element $\rho$ of $\cAH(S)$ form an end structure $\EE(\rho)$ in $\cES(S^-\sqcup S^+)$. For $\rho\in\cAH^+(S)$, its end structure $\EE(\rho)$ has the form $(P,\Sigma^-,m^-,L)\times (\emptyset,S^+,m^+,\emptyset)\in\cES(S^-)\times\cT(S^+)$ (and any such end structure is \emph{doubly incompressible} in the sense of \cite[\S 2.9]{MR3134412}).

Recall that a sequence $(\rho_n)_{n\in\N}$ in $\cAH(S)$ \emph{converges strongly} to $\rho_{\infty}\in\cAH(S)$ if, up to conjugation, it converges algebraically (i.e., $\rho_n\to\rho_{\infty}$ pointwise), and the groups $\rho_n(\Gamma)$ converge geometrically to $\rho_{\infty}(\Gamma)$ (in the Chabauty topology on closed subgroups of $\PSL(2,\C)$), see \cite[\S 3.1]{MR0903850} or \cite[\S 4]{MR3586015}. The topology on $\cAH(S)$ induced by strong convergence is called the \emph{strong topology}, and we write $\cSH(S)$ and $\cSH^+(S)$ for $\cAH(S)$ and $\cAH^+(S)$ endowed with it. The strong topology can also be described in terms of \emph{marked approximate isometries}~\cite[\S 3.2]{MR0903850}, hence it is first-countable.

The following is a special case of Theorem D in \cite{MR3134412}. The injectivity part is a consequence of the Brock--Canary--Minsky ending lamination theorem \cite{MR2925381}, while the continuity and properness parts are due to Anderson and Lecuire \cite{MR3134412}.
\begin{proposition}\label{prop:continuity_end_structures}
	The map $\EE\colon\cSH^+(S)\to\cES(S^-)\times\cT(S^+)$ which associates to $\rho$ its end structure is a homeomorphism, with inverse denoted $\overline{\UU}$.

	In particular, if a sequence $(\rho_n)_{n\in\N}\subset\cQF(S)$ has conformal boundary $(m^-_n,m^+_n)_{n\in\N}\in\cT(S^-)\times\cT(S^+)$ converging to the end structure $\EE(\rho_{\infty})=(\cE^-,m^+)\in\cES(S^-)\times\cT(S^+)$ of $\rho_{\infty}\in\cAH^+(S)$, then $(\rho_n)_{n\in\N}$ converges strongly to $\rho_{\infty}$.
\end{proposition}

\subsection{Quasi-conformal deformations of Kleinian structures}\label{sec:quasi_conformal_deformations_kleinian_groups}
Let $\rho_0\in\cAH(S)$ be a Kleinian surface structure. A \emph{quasi-conformal deformation} of $\rho_0$ is a point $\rho\in\cAH(S)$ for which there exists a quasi-conformal homeomorphism $w\colon S^2\to S^2$ such that $w\circ\rho_0\circ w^{-1}=\rho$. Let $\cQH(\rho_0)\subset\cAH(S)$ denote the subspace of quasi-conformal deformations of $\rho_0$. As defined this way, $\cQH(\rho_0)$ is still endowed with the algebraic topology. However, there is another natural topology on $\cQH(\rho_0)$, where two quasi-conformal deformations are close to each other if they are conjugate by a quasi-conformal homeomorphism with quasi-conformal constant close to $1$. Those two topologies actually coincide: for quasi-Fuchsian groups it follows from Marden's quasi-conformal stability theorem \cite[\S 5.1.2]{MR3586015}, in this case $\cQH(\rho_0)=\cQF(S)$ is the interior of $\cAH(S)$ (if $\rho_0\in\cQF(S)$), and in general from Matsuzaki's \emph{conditional stability} theorem \cite[\S 7.4.2]{MR1638795}. Those topologies on $\cQH(\rho_0)$ also coincide with the topology of strong convergence, since it is in general finer than the algebraic topology and coarser than the quasi-conformal topology.

Combined with the simultaneous uniformization theorem, we have the following result, which follows from the work of Ahlfors, Bers, Kra, Marden, Maskit, Sullivan (see Theorems 7.53 and 7.55 in \cite[\S 7.4.2]{MR1638795}).
\begin{theorem}\label{thm:general_uniformization}
	Let $\rho_0\in\cAH(S)$ with nonempty conformal boundary. Then simultaneous uniformization provides a biholomorphic homeomorphism
	$$U\colon\cT(\partial_{\infty}\rho_0)\overset{\cong}{\longrightarrow}\cQH(\rho_0).$$
	Moreover, $\cQH(\rho_0)$ is a complex (regular) submanifold of an open subset of the representation variety of $\pi_1(S)$ into $\PSL(2,\C)$.
\end{theorem}

Let $e_0=(P_0,L_0)$ be the topological end invariants of $\rho_0\in\cAH(S)$, i.e.\ $P_0$ is its parabolic locus and $L_0$ its ending lamination. Recall that $\cAH(S;e_0)$ denotes the subspace of $\cAH(S)$ consisting of those representations with the same topological end invariants. Then the ending lamination theorem \cite{MR2925381} gives the following.
\begin{theorem}[Ending lamination theorem, Brock--Canary--Minsky \cite{MR2925381}]\label{thm:ending_lamination_theorem}
	Let $\rho_0\in\cAH(S)$ with topological end invariant $e_0$. Then $\cAH(S;e_0)=\cQH(\rho_0)$.
\end{theorem}

\section{Continuity}\label{sec:continuity}
This section focuses on proving the following continuity statement. The proof directly follows from the work of Lecuire~\cite{MR2464096}.
\begin{proposition}\label{prop:continuity_bending_map}
 	The bending map $\bb^+\colon\cSH^+(S)\to\cML(S^+)$ is continuous.
\end{proposition}
\begin{proof}
	The proof is essentially the same as the one of~\cite[Proposition 4.1]{MR2464096}, which proves a more general statement (allowing the emergence of rank one cusps) but in the context of geometrically finite structures. Let us sketch the argument.
	
	Let $(\rho_n)_{n\in\N}$ be a sequence converging to $\rho$ in $\cSH^+(S)$. We shall show that, up to passing to a subsequence, $\bb^+(\rho_n)\to\bb^+(\rho)$ in $\cML(S)$. This can be applied to any subsequence of $(\rho_n)_{n\in\N}$ and is thus sufficient to prove sequential continuity of $\bb^+$, hence continuity since $\cSH^+(S)$ is first-countable. 
	
	Strong convergence implies Hausdorff convergence of their limit sets~\cite{MR1037141,MR1171166}, and of their convex cores~\cite{MR1259606}. Using Arzel\`a--Ascoli theorem and extracting a subsequence, the top boundaries of their convex cores converge as pleated surfaces (in the sense of~\cite{MR2464096}) and remain geometrically finite without cusps by definition of $\cSH^+(S)$. Then \cite[Lemma 3.4]{MR2464096}, which requires strong convergence, concludes.
\end{proof}

\begin{corollary}\label{cor:continuity}
	The map $\BB^+\colon \widehat{\N}\times\cT(S^+)\to\cML(S^+)$ defined in \S \ref{sec:strategy_injectivity} is continuous.
\end{corollary}
\begin{proof}
	By the above Proposition, $\bb^+$ is continuous. The map $\overline{\UU}$ is continuous by Proposition \ref{prop:continuity_end_structures}, while $\overline{r}$ is continuous by construction. Thus $\BB^+$ is continuous.
\end{proof}

\section{Properness}\label{sec:properness}
The goal of this section is to prove that the map $\mm^-\times\bb^+\colon\cAH^+(S;e^-)\to\cT(\Sigma^-)\times\cML^{\perp e^-}_{<\pi}(S^+)$ is proper (see the end of Section \ref{sec:end_structures} for a description of the bending laminations that occur for structures in $\cAH^+(S;e^-)$). This will imply that $\bb^+_{\infty}\colon\cT(S^+)\to\cML^{\perp e^-}_{<\pi}(S^+)$ is proper too. The proof is very similar to the proof of properness of similar maps by Bonahon and Otal in \cite{MR2144972}, Baba and Ohshika in \cite{MR4651897}, Anderson and Lecuire in \cite{MR3134412} and Lecuire in \cite{lecuire2025propernessbendingmap}. It essentially follows by combining the last two references.

\begin{proposition}\label{prop:properness}
	Let $e^-=(P,L)$ be a topological end structure on $S^-$. Let $\Sigma^-\coloneqq S^--(P\cup S(L))$ be the moderate subsurface of the bottom end, and let $S^+=S$ be the one of the top end.

	Let $\cML^{\perp e^-}_{<\pi}(S^+)$ be the subspace of $\cML(S^+)$ consisting of the measured laminations $\mu$ satisfying the following conditions:
	\begin{enumerate}
		\item closed leaves of $\mu$ have weight strictly smaller than $\pi$;
		\item $\mu$ and $e^-$ together fill up the immoderate surface $S(e^-)$ (identifying $S(e^-)$ with a subsurface of $S^+$), i.e.\ for any nonzero measured lamination $\lambda$ on $S(e^-)$ one has $i(\lambda,\mu)+i(\lambda,e^-)>0$ (where in the last inequality one needs to pick a measured lamination representing $e^-$).
	\end{enumerate}
	
	Consider the map
	$$\mm^-\times\bb^+\colon\cAH^+(S;e^-)\longrightarrow\cT(\Sigma^-)\times\cML^{\perp e^-}_{<\pi}(S^+)$$
	that associates to a structure $\rho$ the pair $(\mm^-(\rho),\bb^+(\rho))$ of the conformal structure at infinity of the moderate surface of the bottom end and the bending lamination of the top side of the convex core.

	Then $\mm^-\times\bb^+$ is proper.
\end{proposition}
\begin{proof}
	Consider a sequence of representations $(\rho_n)_n$ in $\cAH^+(S;e^-)$ with convergent ``mixed'' invariants $(\mm^-(\rho_n),\bb^+(\rho_n))_n\subset\cT(\Sigma^-)\times\cML^{\perp e^-}_{<\pi}(S^+)$. By Proposition \ref{prop:algebraic_convergence}, it has an algebraically convergent subsequence. By Corollary \ref{cor:algebraic_limit_in_QH}, this algebraic limit belongs to $\cAH^+(S;e^-)$. This concludes the proof.
\end{proof}

\subsection{Algebraic convergence}
The first step of the proof of properness is the following.

\begin{proposition}\label{prop:algebraic_convergence}
	Let $(\rho_n)_n$ be a sequence of representations in $\cAH^+(S;e^-)$ such that $(\mm^-\times\bb^+(\rho_n))_n$ converges in $\cT(\Sigma^-)\times\cML^{\perp e^-}_{<\pi}(S^+)$. Then $(\rho_n)_n$ has an algebraically convergent subsequence.
\end{proposition}
\begin{proof}
	Write $(\mm^-\times\bb^+)(\rho_n)\eqqcolon(m_n^-,\lambda^+_n)$ for each $n\in\N$, and $(m^-,\lambda^+)$ their limit.

	Let us assume that $(\rho_n)_{n}$ has no algebraically convergent subsequence. By the Morgan--Shalen theorem \ref{thm:Morgan-Shalen} \cite{MR769158} combined with Skora's description of geometric actions on trees \cite{MR1339846} (see 2.2.12 and 2.3.5 in \cite{MR1402300}), up to extracting a subsequence, $(\rho_n)_n$ converges to the action of $\Gamma=\pi_1(S)$ on the real tree $\cT_{\mu}$ dual to a nonzero measured lamination $\mu\in\cML(S)$, i.e.\ there is a sequence of positive real numbers $\epsilon_n\searrow 0$ such that $\epsilon_n\ell_{\rho_n}(\gamma)\to i(\mu,\gamma)$ for all $\gamma\in\Gamma$ (where $\ell_{\rho_n}(\gamma)$ denotes the length of a geodesic representative of $\gamma$ in $M_{\rho_n}\coloneqq\HH^3/\rho_n(\Gamma)$, and is zero if $\rho_n(\gamma)$ is parabolic).

	We will show that $\mu$ cannot intersect $\Sigma^-$ nor $e^-=P\cup L$ nor $\lambda^+$, contradicting the fact that $e^-\cup\lambda^+$ must fill $S(e^-)$. Note that this argument also applies in the case where $\lambda^+=0$, as $\lambda^+=0$ forces $e^-=\emptyset$ and $\Sigma^-=S^-$.

	Since $\epsilon_n\ell_{\rho_n}(c)\to i(\mu,c)$ for any essential closed curve $c$ on $S$ one has the implication
	\begin{equation}\label{eq:curve_length_infinity}
		i(\mu,c)>0\quad\Longrightarrow\quad \ell_{\rho_n}(c)\to \infty.
	\end{equation}

	First, since $P$ consists of parabolics in $M_{\rho_n}$, we have $\ell_{\rho_n}(P)=0$ for each $n$, hence $i(\mu,P)=0$ by \eqref{eq:curve_length_infinity}. Thus, each component of $\mu$ is either contained in $\Sigma^-$, in $S(L)$ or in $P$. Suppose that $\mu$ has a component contained in $\Sigma^-$, then there is an essential closed curve $c$ in $\Sigma^-$ with $i(\mu,c)>0$, hence $(\ell_{\rho_n}(c))_n$ is unbounded by \eqref{eq:curve_length_infinity}. But $(\ell_{m^-_n}(c))_n$ is bounded since $(m^-_n)_n$ is bounded in $\cT(\Sigma^-)$. By Ahlfors' Lemma~\cite[Lemma 5.1.1]{MR1402300}, this implies that $(\ell_{\rho_n}(c))_n$ is bounded, a contradiction. Therefore, $\mu$ has support contained in $S(L)\cup P=S^--\Sigma^-$.

	In \cite{MR3134412}, Anderson and Lecuire show that no component of $L$ can be \emph{realized} in the dual tree $\cT_{\mu}$ (see Section \ref{sec:laminations_and_trees} above, or Claim 3.5 and Lemma 3.6 in \cite{MR3134412}). In particular, by Proposition 3.1.1 in \cite{MR1402300}, we obtain that no component of $L$ \emph{intersects} $\mu$ (in the sense of \S \ref{sec:laminations_and_trees}, see \cite[\S 3]{MR1402300}). Each component of $L$ is arational (i.e.\ minimal, irrational and whose complementary regions are simply connected or annuli) in its respective component of $S(L)$. Thus, each component of $\mu$ which is disjoint from $P$ must have support contained in $L$ \cite[A.3.10]{MR1402300}. In other words, $\mu$ has support contained in $L\cup P=e^-$.

	Since $e^-$ and $\lambda^+$ fill $S(e^-)$ and $\mu$ has support contained in $e^-$, we obtain that $i(\mu,\lambda^+)>0$. Thus there is a minimal component $\lambda^+_0\subset\lambda^+$ with $i(\mu,\lambda^+_0)>0$. In other words, the support of $\lambda^+_0$ is \emph{realized} in $\cT_{\mu}$.

	This leads to a contradiction, as Lecuire shows in \cite{lecuire2025propernessbendingmap} using convergence of the bending laminations to construct sequences of weighted simple closed curves whose lengths with respect to $\rho_n$ grow too slowly, contradicting Otal's continuity theorem (see~\cite[Theorem 4.0.1]{MR1402300},\cite[Theorem 3.7]{MR4264581}).
	
	\begin{theorem}[Otal's continuity theorem]\label{thm:Otal_continuity}
		Let $(\rho_n)_{n\in\N}\subset\cAH(S)$ be a sequence converging to a small minimal action of $\pi_1(S)$ on a real tree $\cT_{\mu}$, where $\mu\in\cML(S)$. Let $\epsilon_n\searrow 0$ be such that $\epsilon_n\ell_{\rho_n}\to i(\mu,-)$. Let $s_0\in\cT(S)$ be any hyperbolic metric on $S$. Let $\alpha\subset S$ be a geodesic lamination \emph{realized} in $\cT_{\mu}$, i.e.\ each leaf of $\alpha$ intersects $\mu$ transversely. Then there exists a neighborhood $\cV(\alpha)\subset\cGL(S)$ of $\alpha$ and constants $K>0$, $n_0$ such that, for any simple closed curve $c\in\cV(\alpha)$ and for any $n\geq n_0$,
		\begin{equation}
			\epsilon_n\ell_{\rho_n}(c)\geq K\ell_{s_0}(c).
		\end{equation}
	\end{theorem}

	The argument leading to the desired contradication, outlined below, is divided into two cases, depending on whether $\lambda^+_0$ is supported on an irrational lamination \cite[\S 3.1]{lecuire2025propernessbendingmap} or a closed curve \cite[\S 3.2]{lecuire2025propernessbendingmap}.

	In the first case, where $\lambda^+_0$ is supported on a minimal irrational lamination, Lecuire constructs a sequence $(\gamma_n)_{n\in\N}$ of weighted simple closed curves $(\gamma_n)_n$ with controlled lengths with respect to $\rho_n$. It is obtained by first approximating the sequence of bending laminations $(\lambda^+_n)_n$ (which have uniformly bounded length $\ell_{\rho_n}(\lambda^+_n)$~\cite{MR1621436}) by weighted simple closed curves $(\nu_n)_n$, so that $\nu_n\to\lambda^+$. Then, for each $n$, $\gamma_n$ is obtained from $\nu_n$ by ``cutting away'' the part of $\nu_n$ that lies outside of $S(\lambda_0^+)$. This can be done in a careful manner so that $(\gamma_n)_n$ converges to a non-trivial measured lamination $\lambda'$ with the same support as $\lambda^+_0$ (up to a subsequence), see \cite[\S 3.1]{lecuire2025propernessbendingmap}. It then follows from the construction that $\epsilon_n\ell_{\rho_n}(\gamma_n)\to 0$, contradicting Theorem \ref{thm:Otal_continuity}. Indeed, since $|\gamma_n|\to|\lambda^+_0|$ in $\cGL(S)$, Theorem \ref{thm:Otal_continuity} implies that, for $n$ large enough, $\epsilon_n\ell_{\rho_n}(\gamma_n)\geq K\ell_{s_0}(\gamma_n)$ but the right-hand side is bounded away from $0$ since $\ell_{s_0}(\gamma_n)\to\ell_{s_0}(\lambda')>0$.

	In the second case, $\lambda_0^+$ is supported on a simple closed curve $d$. In particular, $\epsilon_n\ell_{\rho_n}(d)\to i(\mu,d)>0$. In that case, Lecuire shows that any closed curve $c$ transverse to $d$ satisfies \cite[Lemma 3.5]{lecuire2025propernessbendingmap} $$\frac{\epsilon_n\ell_{\rho_n}(d)}{\epsilon_n\ell_{\rho_n}(c)}=\frac{\ell_{\rho_n}(d)}{\ell_{\rho_n}(c)}\to 0.$$
	But $\epsilon_n\ell_{\rho_n}(c)\to i(\mu,c)$, so $\epsilon_n\ell_{\rho_n}(d)\to 0$, a contradiction. This concludes the proof.
\end{proof}

\subsection{Geometric convergence}
Let $(\rho_n)_n$ be a sequence of representations in $\cAH^+(S;e^-)$ as in Proposition \ref{prop:algebraic_convergence} with $\mm^-\times\bb^+(\rho_n)=(m^-_n,\lambda^+_n)$ converging to $(m^-,\lambda^+)$ in $\cT(\Sigma^-)\times\cML^{\perp e^-}_{<\pi}(S^+)$. By Proposition \ref{prop:algebraic_convergence}, up to replacing $(\rho_n)_n$ by a subsequence, we assume that it converges algebraically to a representation $\rho_{\infty}\in\cAH(S)$. The goal of this section is to show that $\rho_{\infty}$ belongs to $\cAH^+(S;e^-)$, i.e.\ that $\HH^3/\rho_{\infty}(\Gamma)$ has parabolic locus $P$, no new geometrically infinite end, and that it still has $L$ as its ending lamination.

First, the bottom end is controlled.

\begin{lemma}[Anderson--Lecuire]
	The bottom end of the algebraic limit $\HH^3/\rho_{\infty}(\Gamma)$ is geometrically finite on the moderate subsurface $\Sigma^-$, without new rank $1$ cusps on $\Sigma^-$, and has ending lamination $L$ in the immoderate subsurface $S^--(P\cup\Sigma^-)$.
\end{lemma}
\begin{proof}
	The moderate part of the statement is \cite[Lemma 4.1]{MR3134412} and the immoderate part is \cite[Lemma 4.2]{MR3134412}.
\end{proof}

The top end can be controlled thanks to the following result of Bonahon--Otal \cite[\S 5,6]{MR2144972} and Lecuire \cite[\S 4.2]{MR2207784}, \cite[\S 4.2]{lecuire2025propernessbendingmap}. We state a slightly weaker version, although their result also works in presence of parabolics in the top end.

\begin{proposition}[Bonahon--Otal, Lecuire]
	Let $(f_n,\pi_1(S),\rho_n)_n$ be a sequence of convex pleated surfaces such that $(\rho_n)_n$ converges algebraically to some $\rho_{\infty}$. Let $(\lambda_n)_n$ be the sequence of bending laminations of those pleated surfaces. Assume that $\lambda_n\to\lambda$ in $\cML(S)$ and that the closed leaves of $\lambda_n$ and $\lambda$ have weight $<\pi$. Then $(f_n,\pi_1(S),\rho_n)_n$ converges geometrically to a convex pleated surface $(f_{\infty},\pi_1(S),\rho_{\infty})$ inside $\HH^3/\rho_{\infty}(\pi_1(S))$. In particular, the top boundary of the convex core of $\HH^3/\rho_{\infty}(\pi_1(S))$ is homeomorphic to $S^+$.
\end{proposition}

Thus, we obtain the following, which concludes the proof that $\mm^-\times\bb^+$ is proper.
\begin{corollary}\label{cor:algebraic_limit_in_QH}
	The algebraic limit $\HH^3/\rho_{\infty}(\Gamma)$ belongs to $\cAH^+(S;e^-)$.
\end{corollary}

Actually, since $\HH^3/\rho_{\infty}(\pi_1(S))$ has no new parabolics, this also implies that the sequence $(\rho_n)$ converges strongly to $\rho_{\infty}$ (by Anderson--Canary, see \cite[Theorem 4.6.2]{MR3586015}).

\section{Real-analyticity}\label{sec:analyticity}
The following is a variant of~\cite[Lemma 3.1]{dularschlenker2024pleating}. The argument similarly applies Bonahon's shear-bend coordinates~\cite{MR1413855}.
\begin{lemma}\label{lem:real_analyticity}
	Let $e^-=P\cup L$ be a topological end structure on $S^-$ and $(m^-,\lambda^+)\in\cT(\Sigma^-)\times\cML^{\perp e^-}_{<\pi}(S^+)$. Then $(\mm^-\times\bb^+)^{-1}(m^-,\lambda^+)$ is a real-analytic subset of $\cAH^+(S;e^-)$.
\end{lemma}
\begin{proof}
	The fibre $F\coloneqq(\mm^-\times\bb^+)^{-1}(m^-,\lambda^+)$ is the intersection of $(\mm^-)^{-1}(m^-)$ and $(\bb^+)^{-1}(\lambda^+)$, hence it suffices to show that each of these two subsets of $\cAH^+(S;e^-)$ is real-analytic.

	First, Theorems \ref{thm:general_uniformization} and \ref{thm:ending_lamination_theorem} say that $\cAH^+(S;e^-)$ is a complex regular submanifold of an open subset of $\cX(S)$, biholomorphically homeomorphic to $\cT(\Sigma^-)\times\cT(S^+)$. In particular, $(\mm^-)^{-1}(m^-)$ is a complex regular submanifold of an open subset of $\cX(S)$, biholomorphically homeomorphic to $\cT(S^+)$.

	It remains to show that $(\bb^+)^{-1}(\lambda^+)$ is real-analytic. This is a direct consequence of Bonahon's work on the shear-bend coordinates~\cite{MR1413855}, as we shall explain now (the argument is the same as in~\cite{dularschlenker2024pleating}). Pick a maximal geodesic lamination $\Lambda$ containing the support of $\lambda^+$. By Bonahon's shear-bend parameterization Theorem \ref{thm:shear_bend_bonahon}, the open subset $\cR(\Lambda)\subset\cX(S)$ consisting of those representations realizing a pleated surface with pleating locus contained in $\Lambda$ is biholomorphically homeomorphic to an open subset of $\cH(\Lambda;\C/2\pi i\Z)$ via the shear-bend coordinates $\sigma+i\beta$. For a generic $\rho\in\cR(\Lambda)$, the associated pleated surface $P_{\rho}$ lies inside the convex core of $\HH^3/\rho(\pi_1(S))$ and is neither convex nor concave, in particular $\beta(\rho)$ takes positive and negative values in $[-\pi,\pi]$, i.e.\ it does not correspond to a transverse measure. However, if $\rho\in\cR(\Lambda)$ does realize $P_{\rho}$ as the upper boundary of the convex core of $\HH^3/\rho(\pi_1(S))$, then $\beta(\rho)$ corresponds to a transverse measure with support contained in $\Lambda$. In particular, for $\rho_0\in\cAH^+(S;e^-)$ with $\bb^+(\rho_0)=\lambda^+$, the real-analytic subset $\beta^{-1}(\beta(\rho_0))\subset\cR(\Lambda)$ consists of those representations realizing a pleated surface with the same bending cocycle as $\rho_0$, that is, in this case, a convex pleated surface with bending lamination $\lambda^+$. In other words, $(\bb^+)^{-1}(\lambda^+)\cap\cR(\Lambda)=\beta^{-1}(\beta(\rho_0))\cap\cAH^+(S;e^-)\cap\cR(\Lambda)$. Since $\cR(\Lambda)$ is open in $\cX(S)$, this shows that $(\bb^+)^{-1}(\lambda^+)$ is locally real-analytic around $\rho_0$. Thus it is real-analytic.
\end{proof}

In fact, $\cAH^+(S;e^-)\subset\cR(\Lambda)$ by \cite[Proposition 9.7.1]{thurston1979geometry} and \cite[Proposition 5.1]{MR847953}: a geodesic lamination is realized unless it contains a component of the parabolic locus or an ending lamination. This fact was extended to the general case (where compressible ends are allowed) as part of the proof of the density theorem by Namazi--Souto \cite{MR3001608} and Ohshika \cite{MR2821565}.

For our purpose, the fact that the fibres of $\mm^-\times\bb^+$ are real-analytic is important for two reasons, recollected in the following two results.

\begin{lemma}\label{lem:analytic_is_ANR}
	Compact real-analytic spaces are absolute neighborhood retracts (see \S \ref{sec:fibres_of_limit}).
\end{lemma}
\begin{proof}
	Combine \cite{MR278333} with \cite[Corollary 8A]{MR0872468} or see \cite[Lemma 4.8]{dularschlenker2024pleating}.
\end{proof}

\begin{theorem}[Borel--Haefliger \cite{MR149503}, Sullivan \cite{MR278333}]\label{thm:fundamental_class}
	A compact real-analytic space $X$ has a fundamental class modulo two, i.e.\ a nonzero homology class in $H_{\dim(X)}(X;\Z/2)$.
\end{theorem}

\section{Injectivity}\label{sec:injectivity}
As mentioned in the introduction, a key ingredient in showing the main result is to first show the following partial result, where the parabolic locus and ending lamination are empty.

\begin{theorem}\label{thm:prescribed_metric_and_bending}
	Consider the map $\mm^-\times\bb^+\colon\cQF(S)\to\cT(S^-)\times\cML_{<\pi}(S^+)$ sending a quasi-Fuchsian representation $\rho$ to the pair $(\mm^-(\rho),\bb^+(\rho))$ of the conformal structure at infinity of the bottom end and the pleating measured lamination of the top boundary of the convex core of $\HH^3/\rho(\pi_1(S))$. It is a homeomorphism.
\end{theorem}
The proof is very similar to the proof of the injectivity of $\bb^-\times\bb^+$ \cite{dularschlenker2024pleating}. After proving it, we conclude the proof of the main result.

\subsection{Prescribed third fundamental form}
Given $\rho\in\cQF(S)$, a theorem of Labourie \cite{MR1163450} states that each end of $\HH^3/\rho(\pi_1(S))$ is uniquely foliated by smooth closed subsurfaces of constant curvature equal to $K$, with $K$ contained in the interval $(-1,0)$. The third fundamental form of such a subsurface has constant curvature $K^*\coloneqq K/(K+1)$. Thus it gives a point in the \emph{rescaled Teichm\"uller space} $\cT^{K^*}(S)$, which is the space of isotopy classes of metrics of constant curvature $K^*$ on $S$. Focusing on the top end, it follows that, for each $K\in (-1,0)$, one can define the map
\begin{equation*}
	\III^+_K\colon\cQF(S)\to\cT^{K^*}(S)
\end{equation*}
sending a quasi-Fuchsian representation $\rho$ to the third fundamental form of the constant curvature $K$ surface in the top end of $\HH^3/\rho(\pi_1(S))$. The third fundamental form $\III^+_K$ is a ``smooth measure of bending'' and Chen--Schlenker proved that the analogous ``smooth version'' of Theorem \ref{thm:prescribed_metric_and_bending} holds (see Corollary 1.4 and the proof of Lemma 4.1 in \cite{MR4907954}). We need the following special case of their more general result.

\begin{theorem}[Chen--Schlenker \cite{MR4907954}]
	The map $\mm^-\times\III^+_K\colon\cQF(S)\to\cT(S^-)\times\cT^{K^*}(S^+)$ is a homeomorphism.
\end{theorem}

This is a mixed version of the earlier result of Schlenker that $\III^-_K\times\III^+_K$ is a homeomorphism \cite{MR2208419}, which was an essential ingredient in \cite{dularschlenker2024pleating}.

These smooth analogues of $\mm^-\times\bb^+$ actually approximate $\mm^-\times\bb^+$, as follows from the following result of Bonsante--Mondello--Schlenker \cite{MR3035326} and Belraouti \cite{MR3704814}. See the version stated in \cite{dularschlenker2024pleating}.
\begin{proposition}
	When $K\to -1$, the maps $\III^+_K$ converge to $\bb^+$ in the length spectrum, uniformly on compacts. More precisely, $\ell_{\III^+_K(\rho)}(c)\to i(\bb^+(\rho),c)$ when $K\to -1$, for each closed curve $c$ in $S$ and each $\rho\in\cQF(S)$, and the convergence is uniform for $\rho$ in a compact subset of $\cQF(S)$.
\end{proposition}

In order to apply the variant of Finney's theorem as in the next section, we identify the targets of the maps $\III^+_K$, for $K\in (-1,0)$, and $\bb^+$, via the earthquake map $E\colon\cML(S)\times\cT(S)\to\cT(S)\colon (\lambda,m)\mapsto E(\lambda,m)$ as in \cite[\S 4.1]{dularschlenker2024pleating}. More precisely, we fix a point $m_0\in\cT(S)$ and consider, for each $K\in (-1,0)$, the map
\begin{equation}
	V_K\colon\cQF(S)\to\cML(S^+)\colon \rho\mapsto\frac{1}{\sqrt{\left|K^*\right|}}E(-,m_0)^{-1}(\left|K^*\right|\III^+_K(\rho)),
\end{equation}
which is such that $\mm^-\times V_K$ is a homeomorphism since the earthquake map is a homeomorphism by Thurston's earthquake theorem \cite{MR903860}.
Using the estimate \cite[Lemma 4.4]{dularschlenker2024pleating}, we obtain the following.
\begin{corollary}[\!\!{\cite[Lemma 4.2]{dularschlenker2024pleating}}]
	The maps $V_K\colon\cQF(S)\to\cML(S^+)$ converge to $\bb^+$ uniformly on compacts, when $K\to -1$ (with respect to the induced metric when embedding $\cML(S^+)$ inside $\R^{\{c_1,\dots,c_N\}}$ for a large enough collection of closed curves).
\end{corollary}

\subsection{Fibres of a limit of homeomorphisms}\label{sec:fibres_of_limit}
A subset $F\subset N$ is a \emph{neighborhood retract} if it has an open neighborhood $\Omega\subset N$ such that the inclusion $j\colon F\monic\Omega$ has a retraction $r\colon\Omega\to F$, i.e.\ $r\circ j=\id_{F}$. A metrizable space $X$ is an \emph{absolute neighborhood retract} (ANR) if, for any closed embedding $X\monic Y$ into a metrizable space $Y$, $X$ is a neighborhood retract in $Y$.

A sequence of maps $(f_n\colon X\to Y)_{n\in\N}$ \emph{converges continuously} to $f_{\infty}$ if $f_n(x_n)\to f_{\infty}(x)$ whenever $x_n\to x$. When $Y$ is metrizable and $f_{\infty}$ is continuous, continuous convergence coincides with \emph{compact convergence}, i.e.\ uniform convergence on compact subsets of $X$ \cite[p.\ 98]{MR1084167}.

The following result is a variant of the theorem of Finney \cite{MR0224087} used in \cite[Theorem 4.6]{dularschlenker2024pleating}.

\begin{theorem}\label{thm:fibres_are_weakly_contractible}
	Let $N$ and $N'$ be two manifolds of the same dimension. Let $(f_n\colon N\to N')_{n\in\N}$ be a sequence of open embeddings converging continuously (equivalently, compactly) to the continuous map $f_{\infty}\colon N\to N'$, i.e.\ $f_n(x_n)\to f_{\infty}(x_{\infty})$ whenever $x_n\to x_{\infty}$. If a fibre of $f_{\infty}$ is compact and is an absolute neighborhood retract, then it is contractible.
\end{theorem}
\begin{proof}
	Let $F\coloneqq f_{\infty}^{-1}(\ast)$ be a (nonempty) compact fibre, which is a retract of an open neighborhood $\Omega$, which we can assume to have compact closure. Let $j\colon F\monic\Omega$ be the inclusion and $r\colon\Omega\to F$ a retraction. Assuming that $F$ is not contractible, we will show that $f_{\infty}$ sends a point of $\partial\Omega$ to $\ast$.

	Let $U\subset N'$ be a contractible neighborhood of $\ast$. Since $F$ is compact, we have $f_n(F)\subset U$ for all $n$ large enough. Fix such an $n\in\N$. Thus $f_n(F)\subset U$. Since $U$ is contractible, there is a homotopy $H\colon F\times [0,1]\to U$ from $H(-,0)=f_n\restr{F}\colon F\to U$ to a constant map $H(-,1)$. But recall that $f_n$ is an open embedding. Thus $H(F\times [0,1])$ cannot be entirely contained in $f_n(\Omega)$, for otherwise $r\circ (f_n\restr{\Omega})^{-1}\circ H\colon F\times [0,1]\to F$ would be a nullhomotopy of $\id_F$.
	
	Since $H(F\times [0,1])$ is connected ($H$ is a nullhomotopy, hence every point in its image is connected by a path to the constant image of $H(-,1)$), it must intersect $f_n(\partial\Omega)$ (which equals $\partial f_n(\Omega)$ since $f_n$ is an open embedding and $\overline{\Omega}$ is compact), say at a point $f_n(x_U)$ with $x_U\in\partial\Omega$. In particular, $f_n(x_U)\in U$.

	Considering a neighborhood basis $U_0\supset U_1\supset\dots$ of $\ast$ consisting of contractible open sets, the above paragraph gives a sequence $(n_i)_{i\in\N}$ in $\N$, with $n_i\to\infty$, and a sequence of points $(x_i)_{i\in\N}$ in $\partial\Omega$, such that $f_{n_i}(x_i)\in U_i$ for all $i\in\N$.

	By compactness of $\partial\Omega$, up to extracting a subsequence, $(x_i)_{i\in\N}$ converges to some $x_{\infty}\in\partial\Omega$. By continuous convergence, we obtain
	$$f_{\infty}(x_{\infty})=\lim_{i\to\infty} f_{n_i}(x_i)=\ast,$$
	which is impossible since $\partial\Omega$ does not meet $F=f_{\infty}^{-1}(\ast)$. This concludes the proof that $F$ is contractible.
\end{proof}

We are finally able to prove Theorem \ref{thm:prescribed_metric_and_bending}, i.e.\ that $\mm^-\times\bb^+\colon\cQF(S)\to\cT(S^-)\times\cML_{<\pi}(S^+)$ is a homeomorphism, as in \cite{dularschlenker2024pleating} for $\bb^-\times\bb^+$.
\begin{proof}[Proof of Theorem \ref{thm:prescribed_metric_and_bending}]
	First, using Proposition \ref{prop:properness} with empty parabolic locus and ending lamination, the map $\mm^-\times\bb^+$ is proper. In particular, its fibres are compact. Let $F$ be a fibre of $\mm^-\times\bb^+$.

	By Lemmas \ref{lem:real_analyticity} and \ref{lem:analytic_is_ANR}, $F$ is a real-analytic space and an ANR.

	Since $\mm^-\times\bb^+$ is the limit of the homeomorphisms $\mm^-\times V_K$, when $K$ tends to $-1$, and the convergence is uniform on compacts, $F$ is contractible by Theorem \ref{thm:fibres_are_weakly_contractible}.

	Since $F$ is a compact real-analytic space, it has a fundamental class modulo two (by Borel--Haefliger and Sullivan, see Theorem \ref{thm:fundamental_class}), i.e.\ a nonzero homology class $\alpha\in H_{\dim(F)}(F;\Z/2)$. But $F$ is contractible, hence its homology is trivial in all degrees except $0$. Since $\alpha\neq 0$, we must have $\dim(F)=0$. Therefore, $F$ is a contractible compact real-analytic space of dimension $0$, that is, a point. This proves that $\mm^-\times\bb^+$ is injective. By invariance of domain, it is an open embedding. This concludes, since a proper open embedding into a connected space ($\cML_{<\pi}(S^+)$ is star-shaped about $0$, hence connected) is a homeomorphism.
\end{proof}

\subsection{Conclusion}
We now have all the required ingredients to conclude the proof of the main result, following the strategy presented in Section~\ref{sec:strategy_injectivity}.

\begin{proof}[Proof of the main theorem]
The map
$$\mm^-\times\bb^+\colon\cAH^+(S;e^-)\to\cT(\Sigma^-)\times\cML^{\perp e^-}_{<\pi}(S^+)$$
is proper (Proposition \ref{prop:properness}), continuous (on the first factor, by the simultaneous uniformization theorem, see Section \ref{sec:quasi_conformal_deformations_kleinian_groups}; on the second factor, by Proposition \ref{prop:continuity_bending_map} since the algebraic and strong topologies coincide on $\cAH^+(S;e^-)$). Its domain and target are manifolds of the same dimension (by \cite[Theorem 4.7]{MR1150583}, $\cML^{\perp e^-}_{<\pi}(S^+)$ is a manifold of the same dimension as $\cT(S^+)$).

Moreover, the target is connected: $\cT(\Sigma^-)$ is connected, and $\cML^{\perp e^-}_{<\pi}(S^+)$ is the image of the connected space $\cAH^+(S;e^-)\cong\cT(\Sigma^-)\times\cT(S^+)$ under the continuous map $\bb^+$ (Proposition \ref{prop:continuity_bending_map}), which is surjective by the realization theorem of Baba--Ohshika~\cite{MR4651897}.

Thus, it remains to show that it is injective. With the notation of Section \ref{sec:strategy_injectivity}, it suffices to show that $\bb^+_{\infty}$ is injective. But it is the limit of the open embeddings $\bb^+_n$, when $n\to\infty$, and the convergence is continuous since $\BB^+$ is continuous. By following the proof of Theorem \ref{thm:prescribed_metric_and_bending} and using Theorem \ref{thm:fibres_are_weakly_contractible}, we obtain that $\bb^+_{\infty}$ is injective, hence $\bb^+$ is injective too. By invariance of domain it is an open embedding, and a proper open embedding into a connected space is a homeomorphism.
\end{proof}

\bibliographystyle{alpha}
\bibliography{bib.bib}
\end{document}

%% file: commands.tex
\newcommand{\III}{I\hspace{-0.1cm}I\hspace{-0.1cm}I}

\usepackage{amsthm}
\newtheorem{theorem}[subsection]{Theorem}
\newtheorem*{theorem*}{Theorem}
\newtheorem{proposition}[subsection]{\rm\bf Proposition}

\newtheorem{lemma}[subsection]{Lemma}
\newtheorem{corollary}[subsection]{Corollary}

\newtheorem{remark}[subsection]{Remark}

\newtheoremstyle{named}{}{}{\itshape}{}{\bfseries}{.}{.5em}{#1 \thmnote{#3}}
\theoremstyle{named}

\newcommand{\C}{{\mathbb{C}}}

\newcommand{\N}{{\mathbb{N}}}

\newcommand{\HH}{{\mathbb{H}}}
\newcommand{\R}{{\mathbb{R}}}

\newcommand{\Z}{{\mathbb{Z}}}

\newcommand{\cF}{{\mathcal{F}}}
\newcommand{\cE}{{\mathcal{E}}}
\newcommand{\cH}{{\mathcal{H}}}

\newcommand{\cSH}{{\mathcal{SH}}}

\newcommand{\cAH}{{\mathcal{AH}}}
\newcommand{\cQH}{{\mathcal{QH}}}

\newcommand{\cR}{{\mathcal{R}}}
\newcommand{\cT}{{\mathcal{T}}}
\newcommand{\cV}{{\mathcal{V}}}
\newcommand{\cX}{{\mathcal{X}}}

\newcommand{\cML}{{\mathcal{ML}}}
\newcommand{\cGL}{{\mathcal{GL}}}

\newcommand{\cES}{{\mathcal{ES}}}
\newcommand{\cQF}{{\mathcal{QF}}}

\newcommand{\Isom}{\rm{Isom}}
\newcommand{\PSL}{\rm{PSL}}

\DeclareMathOperator{\id}{id}

\DeclareMathOperator{\mm}{m}
\DeclareMathOperator{\bb}{b}

\DeclareMathOperator{\EE}{E}
\DeclareMathOperator{\BB}{B}
\DeclareMathOperator{\UU}{U}

\newcommand{\monic}{\xhookrightarrow{}}

\newcommand{\restr}[1]{|_{#1}}

\newcounter{notes}%

%% file: bib.bib
@article {MR2144972,
    AUTHOR = {Bonahon, Francis and Otal, Jean-Pierre},
     TITLE = {Laminations measur\'ees de plissage des vari\'et\'es
              hyperboliques de dimension 3},
   JOURNAL = {Ann. of Math. (2)},
  FJOURNAL = {Annals of Mathematics. Second Series},
    VOLUME = {160},
      YEAR = {2004},
    NUMBER = {3},
     PAGES = {1013--1055},
      ISSN = {0003-486X,1939-8980},
   MRCLASS = {57M50 (57N10 57R30)},
  MRNUMBER = {2144972},
MRREVIEWER = {Thilo\ Kuessner},
       DOI = {10.4007/annals.2004.160.1013},
       URL = {https://doi.org/10.4007/annals.2004.160.1013},
}

@article{dularschlenker2024pleating,
  title={Convex co-compact hyperbolic manifolds are determined by their pleating lamination},
  author={Dular, Bruno and Schlenker, Jean-Marc},
  journal={To appear in \emph{Ann. of Math. (2)}, arXiv preprint arXiv:2403.10090},
  year={2024}
}

@article {MR3134412,
    AUTHOR = {Anderson, James W. and Lecuire, Cyril},
     TITLE = {Strong convergence of {K}leinian groups: the cracked eggshell},
   JOURNAL = {Comment. Math. Helv.},
  FJOURNAL = {Commentarii Mathematici Helvetici. A Journal of the Swiss
              Mathematical Society},
    VOLUME = {88},
      YEAR = {2013},
    NUMBER = {4},
     PAGES = {813--857},
      ISSN = {0010-2571,1420-8946},
   MRCLASS = {57M50 (30F40)},
  MRNUMBER = {3134412},
MRREVIEWER = {Ken-ichi\ Ohshika},
       DOI = {10.4171/CMH/304},
       URL = {https://doi.org/10.4171/CMH/304},
}

@article {MR2464096,
    AUTHOR = {Lecuire, Cyril},
     TITLE = {Continuity of the bending map},
   JOURNAL = {Ann. Fac. Sci. Toulouse Math. (6)},
  FJOURNAL = {Annales de la Facult\'{e} des Sciences de Toulouse.
              Math\'{e}matiques. S\'{e}rie 6},
    VOLUME = {17},
      YEAR = {2008},
    NUMBER = {1},
     PAGES = {93--119},
      ISSN = {0240-2963,2258-7519},
   MRCLASS = {57M50},
  MRNUMBER = {2464096},
MRREVIEWER = {Thilo\ Kuessner},
       URL = {http://afst.cedram.org/item?id=AFST_2008_6_17_1_93_0},
}

@article {MR1331998,
    AUTHOR = {Keen, Linda and Series, Caroline},
     TITLE = {Continuity of convex hull boundaries},
   JOURNAL = {Pacific J. Math.},
  FJOURNAL = {Pacific Journal of Mathematics},
    VOLUME = {168},
      YEAR = {1995},
    NUMBER = {1},
     PAGES = {183--206},
      ISSN = {0030-8730,1945-5844},
   MRCLASS = {30F40 (57M50)},
  MRNUMBER = {1331998},
MRREVIEWER = {Bruno\ P.\ Zimmermann},
       URL = {http://projecteuclid.org/euclid.pjm/1102620682},
}

@book{thurston1979geometry,
  title={The geometry and topology of three-manifolds},
  author={Thurston, William P.},
  year={1979},
  publisher={Princeton University}
}

@incollection {MR0903850,
    AUTHOR = {Canary, R. D. and Epstein, D. B. A. and Green, P.},
     TITLE = {Notes on notes of {T}hurston},
 BOOKTITLE = {Analytical and geometric aspects of hyperbolic space
              ({C}oventry/{D}urham, 1984)},
    SERIES = {London Math. Soc. Lecture Note Ser.},
    VOLUME = {111},
     PAGES = {3--92},
 PUBLISHER = {Cambridge Univ. Press, Cambridge},
      YEAR = {1987},
      ISBN = {0-521-33906-5},
   MRCLASS = {57N10 (32G15 32G99 54A20 57M99 58F17)},
  MRNUMBER = {903850},
MRREVIEWER = {William\ Dunbar},
}

@book{MR1144770,
    AUTHOR = {Penner, R. C. and Harer, J. L.},
     TITLE = {Combinatorics of train tracks},
    SERIES = {Annals of Mathematics Studies},
    VOLUME = {125},
 PUBLISHER = {Princeton University Press, Princeton, NJ},
      YEAR = {1992},
     PAGES = {xii+216},
      ISBN = {0-691-08764-4; 0-691-02531-2},
   MRCLASS = {57M99 (30F60 57N05 57R30)},
  MRNUMBER = {1144770},
MRREVIEWER = {Y.\ Minsky},
       DOI = {10.1515/9781400882458},
       URL = {https://doi.org/10.1515/9781400882458},
}

@book{MR0872468,
    AUTHOR = {Daverman, Robert J.},
     TITLE = {Decompositions of manifolds},
    SERIES = {Pure and Applied Mathematics},
    VOLUME = {124},
 PUBLISHER = {Academic Press, Inc., Orlando, FL},
      YEAR = {1986},
     PAGES = {xii+317},
      ISBN = {0-12-204220-4},
   MRCLASS = {57-01 (54B15)},
  MRNUMBER = {872468},
}

@article {MR0224087,
    AUTHOR = {Finney, Ross L.},
     TITLE = {Pseudo-isotopies and cellular sets},
   JOURNAL = {Michigan Math. J.},
  FJOURNAL = {Michigan Mathematical Journal},
    VOLUME = {14},
      YEAR = {1967},
     PAGES = {417--421},
      ISSN = {0026-2285,1945-2365},
   MRCLASS = {55.30},
  MRNUMBER = {224087},
MRREVIEWER = {A.\ H.\ Copeland, Jr.},
       URL = {http://projecteuclid.org/euclid.mmj/1028999842},
}

@article {MR1413855,
    AUTHOR = {Bonahon, Francis},
     TITLE = {Shearing hyperbolic surfaces, bending pleated surfaces and
              {T}hurston's symplectic form},
   JOURNAL = {Ann. Fac. Sci. Toulouse Math. (6)},
  FJOURNAL = {Toulouse. Facult\'e{} des Sciences. Annales. Math\'ematiques.
              S\'erie 6},
    VOLUME = {5},
      YEAR = {1996},
    NUMBER = {2},
     PAGES = {233--297},
      ISSN = {0240-2955},
   MRCLASS = {57M50 (53C15 57N05 57N10)},
  MRNUMBER = {1413855},
MRREVIEWER = {Athanase\ Papadopoulos},
       URL = {http://www.numdam.org/item?id=AFST_1996_6_5_2_233_0},
}

@article {MR4651897,
    AUTHOR = {Baba, Shinpei and Ohshika, Ken'ichi},
     TITLE = {Realisation of bending measured laminations by {K}leinian
              surface groups},
   JOURNAL = {Int. Math. Res. Not. IMRN},
  FJOURNAL = {International Mathematics Research Notices. IMRN},
      YEAR = {2023},
    NUMBER = {19},
     PAGES = {16674--16707},
      ISSN = {1073-7928,1687-0247},
   MRCLASS = {57K20 (30F40)},
  MRNUMBER = {4651897},
MRREVIEWER = {Hongbin\ Sun},
       DOI = {10.1093/imrn/rnac357},
       URL = {https://doi.org/10.1093/imrn/rnac357},
}

@inproceedings {MR278333,
    AUTHOR = {Sullivan, D.},
     TITLE = {Combinatorial invariants of analytic spaces},
 BOOKTITLE = {Proceedings of {L}iverpool {S}ingularities---{S}ymposium, {I}
              (1969/70)},
    SERIES = {Lecture Notes in Math.},
    VOLUME = {Vol. 192},
     PAGES = {165--168},
 PUBLISHER = {Springer, Berlin-New York},
      YEAR = {1971},
   MRCLASS = {57.36 (32.00)},
  MRNUMBER = {278333},
MRREVIEWER = {A.\ H.\ Wallace},
}

@article {MR149503,
    AUTHOR = {Borel, Armand and Haefliger, Andr\'e},
     TITLE = {La classe d'homologie fondamentale d'un espace analytique},
   JOURNAL = {Bull. Soc. Math. France},
  FJOURNAL = {Bulletin de la Soci\'et\'e{} Math\'ematique de France},
    VOLUME = {89},
      YEAR = {1961},
     PAGES = {461--513},
      ISSN = {0037-9484},
   MRCLASS = {57.32 (32.47)},
  MRNUMBER = {149503},
MRREVIEWER = {E.\ Brieskorn},
       URL = {http://www.numdam.org/item?id=BSMF_1961__89__461_0},
}

@article {MR1150583,
    AUTHOR = {Kerckhoff, Steven P.},
     TITLE = {Lines of minima in {T}eichm\"uller space},
   JOURNAL = {Duke Math. J.},
  FJOURNAL = {Duke Mathematical Journal},
    VOLUME = {65},
      YEAR = {1992},
    NUMBER = {2},
     PAGES = {187--213},
      ISSN = {0012-7094,1547-7398},
   MRCLASS = {32G15 (30F60 57N05)},
  MRNUMBER = {1150583},
MRREVIEWER = {Colette\ Ann\'e},
       DOI = {10.1215/S0012-7094-92-06507-0},
       URL = {https://doi.org/10.1215/S0012-7094-92-06507-0},
}

@article {MR1621436,
    AUTHOR = {Bridgeman, Martin},
     TITLE = {Average bending of convex pleated planes in hyperbolic
              three-space},
   JOURNAL = {Invent. Math.},
  FJOURNAL = {Inventiones Mathematicae},
    VOLUME = {132},
      YEAR = {1998},
    NUMBER = {2},
     PAGES = {381--391},
      ISSN = {0020-9910,1432-1297},
   MRCLASS = {57M50 (30F40 30F50 53A35)},
  MRNUMBER = {1621436},
MRREVIEWER = {Michel\ Coornaert},
       DOI = {10.1007/s002220050227},
       URL = {https://doi.org/10.1007/s002220050227},
}

@article {MR1402300,
    AUTHOR = {Otal, Jean-Pierre},
     TITLE = {Le th\'eor\`eme d'hyperbolisation pour les vari\'et\'es
              fibr\'ees de dimension 3},
   JOURNAL = {Ast\'erisque},
  FJOURNAL = {Ast\'erisque},
    NUMBER = {235},
      YEAR = {1996},
     PAGES = {x+159},
      ISSN = {0303-1179,2492-5926},
   MRCLASS = {57M50 (20H10 30F40)},
  MRNUMBER = {1402300},
MRREVIEWER = {Lee\ Mosher},
}

@article {MR769158,
    AUTHOR = {Morgan, John W. and Shalen, Peter B.},
     TITLE = {Valuations, trees, and degenerations of hyperbolic structures.
              {I}},
   JOURNAL = {Ann. of Math. (2)},
  FJOURNAL = {Annals of Mathematics. Second Series},
    VOLUME = {120},
      YEAR = {1984},
    NUMBER = {3},
     PAGES = {401--476},
      ISSN = {0003-486X,1939-8980},
   MRCLASS = {57N10 (14M99 32G15)},
  MRNUMBER = {769158},
MRREVIEWER = {Christopher\ W.\ Stark},
       DOI = {10.2307/1971082},
       URL = {https://doi.org/10.2307/1971082},
}

@article {MR1339846,
    AUTHOR = {Skora, Richard K.},
     TITLE = {Splittings of surfaces},
   JOURNAL = {J. Amer. Math. Soc.},
  FJOURNAL = {Journal of the American Mathematical Society},
    VOLUME = {9},
      YEAR = {1996},
    NUMBER = {2},
     PAGES = {605--616},
      ISSN = {0894-0347,1088-6834},
   MRCLASS = {57M50 (20F32 57M05 57R30)},
  MRNUMBER = {1339846},
MRREVIEWER = {William\ Dunbar},
       DOI = {10.1090/S0894-0347-96-00205-6},
       URL = {https://doi.org/10.1090/S0894-0347-96-00205-6},
}

@incollection {MR4264581,
    AUTHOR = {Lecuire, Cyril},
     TITLE = {The double limit theorem and its legacy},
 BOOKTITLE = {In the tradition of {T}hurston---geometry and topology},
     PAGES = {263--290},
 PUBLISHER = {Springer, Cham},
      YEAR = {2020},
      ISBN = {978-3-030-55928-1; 978-3-030-55927-4},
   MRCLASS = {57M30 (30F40 30F60)},
  MRNUMBER = {4264581},
       DOI = {10.1007/978-3-030-55928-1\_8},
       URL = {https://doi.org/10.1007/978-3-030-55928-1_8},
}

@article {MR2207784,
    AUTHOR = {Lecuire, Cyril},
     TITLE = {Plissage des vari\'et\'es hyperboliques de dimension 3},
   JOURNAL = {Invent. Math.},
  FJOURNAL = {Inventiones Mathematicae},
    VOLUME = {164},
      YEAR = {2006},
    NUMBER = {1},
     PAGES = {85--141},
      ISSN = {0020-9910,1432-1297},
   MRCLASS = {57N10 (57M50)},
  MRNUMBER = {2207784},
MRREVIEWER = {Thilo\ Kuessner},
       DOI = {10.1007/s00222-005-0470-z},
       URL = {https://doi.org/10.1007/s00222-005-0470-z},
}

@book {MR3586015,
    AUTHOR = {Marden, Albert},
     TITLE = {Hyperbolic manifolds},
      NOTE = {An introduction in 2 and 3 dimensions},
 PUBLISHER = {Cambridge University Press, Cambridge},
      YEAR = {2016},
     PAGES = {xviii+515},
      ISBN = {978-1-107-11674-0},
   MRCLASS = {57-01 (20H10 30F40 57M50 57N05 57N10)},
  MRNUMBER = {3586015},
MRREVIEWER = {Thilo\ Kuessner},
       DOI = {10.1017/CBO9781316337776},
       URL = {https://doi.org/10.1017/CBO9781316337776},
}

@article {MR4907954,
    AUTHOR = {Chen, Qiyu and Schlenker, Jean-Marc},
     TITLE = {The geometric data on the boundary of convex subsets of
              hyperbolic manifolds},
   JOURNAL = {Trans. Amer. Math. Soc.},
  FJOURNAL = {Transactions of the American Mathematical Society},
    VOLUME = {378},
      YEAR = {2025},
    NUMBER = {6},
     PAGES = {4225--4301},
      ISSN = {0002-9947,1088-6850},
   MRCLASS = {58J32 (57K32 58J05)},
  MRNUMBER = {4907954},
MRREVIEWER = {Thilo\ Kuessner},
       DOI = {10.1090/tran/9391},
       URL = {https://doi.org/10.1090/tran/9391},
}

@article {MR2208419,
    AUTHOR = {Schlenker, Jean-Marc},
     TITLE = {Hyperbolic manifolds with convex boundary},
   JOURNAL = {Invent. Math.},
  FJOURNAL = {Inventiones Mathematicae},
    VOLUME = {163},
      YEAR = {2006},
    NUMBER = {1},
     PAGES = {109--169},
      ISSN = {0020-9910,1432-1297},
   MRCLASS = {57M50 (53C21)},
  MRNUMBER = {2208419},
MRREVIEWER = {Igor\ Rivin},
       DOI = {10.1007/s00222-005-0456-x},
       URL = {https://doi.org/10.1007/s00222-005-0456-x},
}

@article {MR1163450,
    AUTHOR = {Labourie, Fran\c{c}ois},
     TITLE = {M\'etriques prescrites sur le bord des vari\'et\'es
              hyperboliques de dimension {$3$}},
   JOURNAL = {J. Differential Geom.},
  FJOURNAL = {Journal of Differential Geometry},
    VOLUME = {35},
      YEAR = {1992},
    NUMBER = {3},
     PAGES = {609--626},
      ISSN = {0022-040X,1945-743X},
   MRCLASS = {53C20 (53C21 57M50)},
  MRNUMBER = {1163450},
MRREVIEWER = {Viktor\ Schroeder},
       URL = {http://projecteuclid.org/euclid.jdg/1214448258},
}

@article {MR3035326,
    AUTHOR = {Bonsante, Francesco and Mondello, Gabriele and Schlenker,
              Jean-Marc},
     TITLE = {A cyclic extension of the earthquake flow {I}},
   JOURNAL = {Geom. Topol.},
  FJOURNAL = {Geometry \& Topology},
    VOLUME = {17},
      YEAR = {2013},
    NUMBER = {1},
     PAGES = {157--234},
      ISSN = {1465-3060,1364-0380},
   MRCLASS = {32G15 (30F60 57M50 58H15)},
  MRNUMBER = {3035326},
MRREVIEWER = {Athanase\ Papadopoulos},
       DOI = {10.2140/gt.2013.17.157},
       URL = {https://doi.org/10.2140/gt.2013.17.157},
}

@article {MR3704814,
    AUTHOR = {Belraouti, Mehdi},
     TITLE = {Asymptotic behavior of {C}auchy hypersurfaces in constant
              curvature space-times},
   JOURNAL = {Geom. Dedicata},
  FJOURNAL = {Geometriae Dedicata},
    VOLUME = {190},
      YEAR = {2017},
     PAGES = {103--133},
      ISSN = {0046-5755,1572-9168},
   MRCLASS = {53C50},
  MRNUMBER = {3704814},
MRREVIEWER = {Miguel\ S\'anchez},
       DOI = {10.1007/s10711-017-0230-4},
       URL = {https://doi.org/10.1007/s10711-017-0230-4},
}

@incollection {MR903860,
    AUTHOR = {Thurston, William P.},
     TITLE = {Earthquakes in two-dimensional hyperbolic geometry},
 BOOKTITLE = {Low-dimensional topology and {K}leinian groups
              ({C}oventry/{D}urham, 1984)},
    SERIES = {London Math. Soc. Lecture Note Ser.},
    VOLUME = {112},
     PAGES = {91--112},
 PUBLISHER = {Cambridge Univ. Press, Cambridge},
      YEAR = {1986},
      ISBN = {0-521-33905-7},
   MRCLASS = {57M99 (22E40 30F20)},
  MRNUMBER = {903860},
MRREVIEWER = {N.\ V.\ Ivanov},
}

@article {MR2925381,
    AUTHOR = {Brock, Jeffrey F. and Canary, Richard D. and Minsky, Yair N.},
     TITLE = {The classification of {K}leinian surface groups, {II}: {T}he
              ending lamination conjecture},
   JOURNAL = {Ann. of Math. (2)},
  FJOURNAL = {Annals of Mathematics. Second Series},
    VOLUME = {176},
      YEAR = {2012},
    NUMBER = {1},
     PAGES = {1--149},
      ISSN = {0003-486X,1939-8980},
   MRCLASS = {57M50 (30F40)},
  MRNUMBER = {2925381},
MRREVIEWER = {Athanase\ Papadopoulos},
       DOI = {10.4007/annals.2012.176.1.1},
       URL = {https://doi.org/10.4007/annals.2012.176.1.1},
}

@article {MR3001608,
    AUTHOR = {Namazi, Hossein and Souto, Juan},
     TITLE = {Non-realizability and ending laminations: proof of the density
              conjecture},
   JOURNAL = {Acta Math.},
  FJOURNAL = {Acta Mathematica},
    VOLUME = {209},
      YEAR = {2012},
    NUMBER = {2},
     PAGES = {323--395},
      ISSN = {0001-5962,1871-2509},
   MRCLASS = {30F40 (20H10 57M50)},
  MRNUMBER = {3001608},
MRREVIEWER = {Majid\ Heydarpour},
       DOI = {10.1007/s11511-012-0088-0},
       URL = {https://doi.org/10.1007/s11511-012-0088-0},
}

@article {MR2821565,
    AUTHOR = {Ohshika, Ken'ichi},
     TITLE = {Realising end invariants by limits of minimally parabolic,
              geometrically finite groups},
   JOURNAL = {Geom. Topol.},
  FJOURNAL = {Geometry \& Topology},
    VOLUME = {15},
      YEAR = {2011},
    NUMBER = {2},
     PAGES = {827--890},
      ISSN = {1465-3060,1364-0380},
   MRCLASS = {57M50 (30F40)},
  MRNUMBER = {2821565},
MRREVIEWER = {Bruno\ P.\ Zimmermann},
       DOI = {10.2140/gt.2011.15.827},
       URL = {https://doi.org/10.2140/gt.2011.15.827},
}

@article {MR111834,
    AUTHOR = {Bers, Lipman},
     TITLE = {Simultaneous uniformization},
   JOURNAL = {Bull. Amer. Math. Soc.},
  FJOURNAL = {Bulletin of the American Mathematical Society},
    VOLUME = {66},
      YEAR = {1960},
     PAGES = {94--97},
      ISSN = {0002-9904},
   MRCLASS = {30.00},
  MRNUMBER = {111834},
MRREVIEWER = {H.\ L.\ Royden},
       DOI = {10.1090/S0002-9904-1960-10413-2},
       URL = {https://doi.org/10.1090/S0002-9904-1960-10413-2},
}

@book {MR2230672,
     TITLE = {Fundamentals of hyperbolic geometry: selected expositions},
    SERIES = {London Mathematical Society Lecture Note Series},
    VOLUME = {328},
    EDITOR = {Canary, Richard D. and Epstein, David and Marden, Albert},
 PUBLISHER = {Cambridge University Press, Cambridge},
      YEAR = {2006},
     PAGES = {xii+335},
      ISBN = {0-521-61558-5},
   MRCLASS = {57-06 (57M50)},
  MRNUMBER = {2230672},
       DOI = {10.1017/CBO9781139106986},
       URL = {https://doi.org/10.1017/CBO9781139106986},
}

@book {MR1638795,
    AUTHOR = {Matsuzaki, Katsuhiko and Taniguchi, Masahiko},
     TITLE = {Hyperbolic manifolds and {K}leinian groups},
    SERIES = {Oxford Mathematical Monographs},
      NOTE = {Oxford Science Publications},
 PUBLISHER = {The Clarendon Press, Oxford University Press, New York},
      YEAR = {1998},
     PAGES = {x+253},
      ISBN = {0-19-850062-9},
   MRCLASS = {30F40 (20H10 30F10)},
  MRNUMBER = {1638795},
MRREVIEWER = {I.\ Kra},
}

@book {MR1084167,
    AUTHOR = {Remmert, Reinhold},
     TITLE = {Theory of complex functions},
    SERIES = {Graduate Texts in Mathematics},
    VOLUME = {122},
   EDITION = {German},
      NOTE = {Readings in Mathematics},
 PUBLISHER = {Springer-Verlag, New York},
      YEAR = {1991},
     PAGES = {xx+453},
      ISBN = {0-387-97195-5},
   MRCLASS = {30-01},
  MRNUMBER = {1084167},
MRREVIEWER = {M.\ H.\ Heins},
       DOI = {10.1007/978-1-4612-0939-3},
       URL = {https://doi.org/10.1007/978-1-4612-0939-3},
}

@article {MR847953,
    AUTHOR = {Bonahon, Francis},
     TITLE = {Bouts des vari\'et\'es hyperboliques de dimension {$3$}},
   JOURNAL = {Ann. of Math. (2)},
  FJOURNAL = {Annals of Mathematics. Second Series},
    VOLUME = {124},
      YEAR = {1986},
    NUMBER = {1},
     PAGES = {71--158},
      ISSN = {0003-486X,1939-8980},
   MRCLASS = {57N10 (30F40 32G15)},
  MRNUMBER = {847953},
MRREVIEWER = {G.\ Peter\ Scott},
       DOI = {10.2307/1971388},
       URL = {https://doi.org/10.2307/1971388},
}

@article {MR2079598,
    AUTHOR = {Brock, Jeffrey F. and Bromberg, Kenneth W.},
     TITLE = {On the density of geometrically finite {K}leinian groups},
   JOURNAL = {Acta Math.},
  FJOURNAL = {Acta Mathematica},
    VOLUME = {192},
      YEAR = {2004},
    NUMBER = {1},
     PAGES = {33--93},
      ISSN = {0001-5962,1871-2509},
   MRCLASS = {57M50 (30F40)},
  MRNUMBER = {2079598},
MRREVIEWER = {Lee\ Mosher},
       DOI = {10.1007/BF02441085},
       URL = {https://doi.org/10.1007/BF02441085},
}

@article {MR167618,
    AUTHOR = {Ahlfors, Lars V.},
     TITLE = {Finitely generated {K}leinian groups},
   JOURNAL = {Amer. J. Math.},
  FJOURNAL = {American Journal of Mathematics},
    VOLUME = {86},
      YEAR = {1964},
     PAGES = {413--429},
      ISSN = {0002-9327,1080-6377},
   MRCLASS = {30.49},
  MRNUMBER = {167618},
MRREVIEWER = {M.\ H.\ Heins},
       DOI = {10.2307/2373173},
       URL = {https://doi.org/10.2307/2373173},
}

@article {MR2186972,
    AUTHOR = {Bonahon, Francis},
     TITLE = {Kleinian groups which are almost {F}uchsian},
   JOURNAL = {J. Reine Angew. Math.},
  FJOURNAL = {Journal f\"ur die Reine und Angewandte Mathematik. [Crelle's
              Journal]},
    VOLUME = {587},
      YEAR = {2005},
     PAGES = {1--15},
      ISSN = {0075-4102,1435-5345},
   MRCLASS = {57M50 (30D35 30D40 30F40)},
  MRNUMBER = {2186972},
MRREVIEWER = {Bruno\ P.\ Zimmermann},
       DOI = {10.1515/crll.2005.2005.587.1},
       URL = {https://doi.org/10.1515/crll.2005.2005.587.1},
}

@article {MR2052972,
    AUTHOR = {Keen, Linda and Series, Caroline},
     TITLE = {Pleating invariants for punctured torus groups},
   JOURNAL = {Topology},
  FJOURNAL = {Topology. An International Journal of Mathematics},
    VOLUME = {43},
      YEAR = {2004},
    NUMBER = {2},
     PAGES = {447--491},
      ISSN = {0040-9383},
   MRCLASS = {30F40 (20H10 57M50)},
  MRNUMBER = {2052972},
MRREVIEWER = {John\ R.\ Parker},
       DOI = {10.1016/S0040-9383(03)00052-1},
       URL = {https://doi.org/10.1016/S0040-9383(03)00052-1},
}

@incollection {MR2258745,
    AUTHOR = {Series, Caroline},
     TITLE = {Thurston's bending measure conjecture for once punctured torus
              groups},
 BOOKTITLE = {Spaces of {K}leinian groups},
    SERIES = {London Math. Soc. Lecture Note Ser.},
    VOLUME = {329},
     PAGES = {75--89},
 PUBLISHER = {Cambridge Univ. Press, Cambridge},
      YEAR = {2006},
      ISBN = {978-0-521-61797-0; 0-521-61797-9},
   MRCLASS = {30F40 (20H10 30F35 57M50 57N10)},
  MRNUMBER = {2258745},
MRREVIEWER = {Bruno\ P.\ Zimmermann},
}

@article {MR1791139,
    AUTHOR = {Brock, J. F.},
     TITLE = {Continuity of {T}hurston's length function},
   JOURNAL = {Geom. Funct. Anal.},
  FJOURNAL = {Geometric and Functional Analysis},
    VOLUME = {10},
      YEAR = {2000},
    NUMBER = {4},
     PAGES = {741--797},
      ISSN = {1016-443X,1420-8970},
   MRCLASS = {57M50 (57M27 57N10 57R30)},
  MRNUMBER = {1791139},
MRREVIEWER = {Sandra\ L.\ Shields},
       DOI = {10.1007/PL00001637},
       URL = {https://doi.org/10.1007/PL00001637},
}

@article {MR1037141,
    AUTHOR = {Kerckhoff, Steven P. and Thurston, William P.},
     TITLE = {Noncontinuity of the action of the modular group at {B}ers'
              boundary of {T}eichm\"uller space},
   JOURNAL = {Invent. Math.},
  FJOURNAL = {Inventiones Mathematicae},
    VOLUME = {100},
      YEAR = {1990},
    NUMBER = {1},
     PAGES = {25--47},
      ISSN = {0020-9910,1432-1297},
   MRCLASS = {57S30 (22E40 32G15)},
  MRNUMBER = {1037141},
MRREVIEWER = {Athanase\ Papadopoulos},
       DOI = {10.1007/BF01231179},
       URL = {https://doi.org/10.1007/BF01231179},
}

@misc{lecuire2025propernessbendingmap,
      title={Properness of the bending map}, 
      author={Cyril Lecuire},
      year={2025},
      eprint={2510.07087},
      archivePrefix={arXiv},
      primaryClass={math.GT},
      url={https://arxiv.org/abs/2510.07087},
      note = {arXiv:2510.07087}
}

@article {MR1678469,
    AUTHOR = {Bonahon, Francis},
     TITLE = {Variations of the boundary geometry of {$3$}-dimensional
              hyperbolic convex cores},
   JOURNAL = {J. Differential Geom.},
  FJOURNAL = {Journal of Differential Geometry},
    VOLUME = {50},
      YEAR = {1998},
    NUMBER = {1},
     PAGES = {1--24},
      ISSN = {0022-040X,1945-743X},
   MRCLASS = {57N10 (30F45 53C22 57M50)},
  MRNUMBER = {1678469},
MRREVIEWER = {Dubravko\ Ivan\v si\'c},
       URL = {http://projecteuclid.org/euclid.jdg/1214510044},
}

@article {MR326737,
    AUTHOR = {Scott, G. P.},
     TITLE = {Compact submanifolds of {$3$}-manifolds},
   JOURNAL = {J. London Math. Soc. (2)},
  FJOURNAL = {Journal of the London Mathematical Society. Second Series},
    VOLUME = {7},
      YEAR = {1973},
     PAGES = {246--250},
      ISSN = {0024-6107,1469-7750},
   MRCLASS = {57A10},
  MRNUMBER = {326737},
MRREVIEWER = {C.\ D.\ Feustel},
       DOI = {10.1112/jlms/s2-7.2.246},
       URL = {https://doi.org/10.1112/jlms/s2-7.2.246},
}

@book {MR415619,
    AUTHOR = {Hempel, John},
     TITLE = {{$3$}-{M}anifolds},
    SERIES = {Annals of Mathematics Studies},
    VOLUME = {No. 86},
 PUBLISHER = {Princeton University Press, Princeton, NJ; University of Tokyo
              Press, Tokyo},
      YEAR = {1976},
     PAGES = {xii+195},
   MRCLASS = {57A10},
  MRNUMBER = {415619},
MRREVIEWER = {J.\ S.\ Birman},
}

@article {MR1259606,
    AUTHOR = {Bowditch, B. H.},
     TITLE = {Some results on the geometry of convex hulls in manifolds of
              pinched negative curvature},
   JOURNAL = {Comment. Math. Helv.},
  FJOURNAL = {Commentarii Mathematici Helvetici},
    VOLUME = {69},
      YEAR = {1994},
    NUMBER = {1},
     PAGES = {49--81},
      ISSN = {0010-2571,1420-8946},
   MRCLASS = {53C20 (53C23)},
  MRNUMBER = {1259606},
MRREVIEWER = {Athanase\ Papadopoulos},
       DOI = {10.1007/BF02564474},
       URL = {https://doi.org/10.1007/BF02564474},
}

@article {MR409766,
    AUTHOR = {Keynes, Harvey B. and Newton, Dan},
     TITLE = {A ``minimal'', non-uniquely ergodic interval exchange
              transformation},
   JOURNAL = {Math. Z.},
  FJOURNAL = {Mathematische Zeitschrift},
    VOLUME = {148},
      YEAR = {1976},
    NUMBER = {2},
     PAGES = {101--105},
      ISSN = {0025-5874,1432-1823},
   MRCLASS = {28A65},
  MRNUMBER = {409766},
MRREVIEWER = {Alan\ Saleski},
       DOI = {10.1007/BF01214699},
       URL = {https://doi.org/10.1007/BF01214699},
}

@article {MR391184,
    AUTHOR = {Sataev, E. A.},
     TITLE = {The number of invariant measures for flows on orientable
              surfaces},
   JOURNAL = {Izv. Akad. Nauk SSSR Ser. Mat.},
  FJOURNAL = {Izvestiya Akademii Nauk SSSR. Seriya Matematicheskaya},
    VOLUME = {39},
      YEAR = {1975},
    NUMBER = {4},
     PAGES = {860--878},
      ISSN = {0373-2436},
   MRCLASS = {58F99 (28A65)},
  MRNUMBER = {391184},
MRREVIEWER = {D.\ Newton},
}

@article {MR2470969,
    AUTHOR = {Gabai, David},
     TITLE = {Almost filling laminations and the connectivity of ending
              lamination space},
   JOURNAL = {Geom. Topol.},
  FJOURNAL = {Geometry \& Topology},
    VOLUME = {13},
      YEAR = {2009},
    NUMBER = {2},
     PAGES = {1017--1041},
      ISSN = {1465-3060,1364-0380},
   MRCLASS = {57M50 (20F65 57R30)},
  MRNUMBER = {2470969},
MRREVIEWER = {Baris\ Coskunuzer},
       DOI = {10.2140/gt.2009.13.1017},
       URL = {https://doi.org/10.2140/gt.2009.13.1017},
}

@article {MR1171166,
    AUTHOR = {Ohshika, Ken'ichi},
     TITLE = {Strong convergence of {K}leinian groups and {C}arath\'eodory
              convergence of domains of discontinuity},
   JOURNAL = {Math. Proc. Cambridge Philos. Soc.},
  FJOURNAL = {Mathematical Proceedings of the Cambridge Philosophical
              Society},
    VOLUME = {112},
      YEAR = {1992},
    NUMBER = {2},
     PAGES = {297--307},
      ISSN = {0305-0041,1469-8064},
   MRCLASS = {57M50 (20H10 30F40 30F45)},
  MRNUMBER = {1171166},
MRREVIEWER = {Darryl\ McCullough},
       DOI = {10.1017/S0305004100070985},
       URL = {https://doi.org/10.1017/S0305004100070985},
}
